\documentclass[12pt]{article}
\usepackage{amsmath,amssymb}
\usepackage{amsthm}
\usepackage{epsfig}
\usepackage{graphics}
\usepackage{graphicx}
\usepackage{colordvi}
\usepackage{mathrsfs} 
\usepackage{stmaryrd}
\usepackage{geometry}
\usepackage{dsfont}
\usepackage{tikz}
\usetikzlibrary{automata}
\usepackage{url}
\usepackage{hyperref}
\hypersetup{
   colorlinks=false,
    linkcolor=blue!75!black,
   filecolor=blue!75!black,      
   urlcolor=blue!75!black, 
   citecolor=blue!75!black,
}
\usepackage[textsize=small]{todonotes}
\usepackage[margin=1cm, size=small]{caption}
\usepackage{multirow}
\allowdisplaybreaks[4]
\oddsidemargin
-.5in \evensidemargin -.5in\marginparwidth 0.4in
\usepackage{color,soul}
\usepackage{xcolor}

\usepackage{etoolbox}
\patchcmd{\thebibliography}{\section*}{\section}{}{}
\usepackage{lineno}

\newcommand{\1}{\mathbf 1}

\newcommand{\F}{\mathscr F}

\newcommand{\vep}{{\varepsilon}}

\renewcommand{\P}{{\mathbb P}}
\newcommand{\E}{{\mathbb E}}

\newcommand{\Z}{{\mathbb Z}}
\newcommand{\R}{{\mathbb R}}
\newcommand{\Id}{{\rm Id}}

\renewenvironment{proof}[1][\proofname]{\noindent {\bfseries #1.}\;}{\hfill\ensuremath{\blacksquare}\\}

\usepackage{currfile}
\usepackage{fancyhdr}
\newtheoremstyle{slantthm}{10pt}{10pt}{\slshape}{}{\bfseries}{}{.5em}{\thmname{#1}\thmnumber{ #2}\thmnote{ (#3)}.}
\newtheoremstyle{slantrmk}{10pt}{10pt}{\rmfamily}{}{\bfseries}{}{.5em}{\thmname{#1}\thmnumber{ #2}\thmnote{ (#3)}.}

\begin{document}
\theoremstyle{slantthm}
\newtheorem{thm}{Theorem}[section]
\newtheorem{prop}[thm]{Proposition}
\newtheorem{lem}[thm]{Lemma}
\newtheorem{cor}[thm]{Corollary}
\newtheorem{defi}[thm]{Definition}
\newtheorem{disc}[thm]{Discussion}
\newtheorem*{nota}{Notation}
\newtheorem{conj}[thm]{Conjecture}

\numberwithin{equation}{section}

\theoremstyle{slantrmk}
\newtheorem{ass}[thm]{Assumption}
\newtheorem{rmk}[thm]{Remark}
\newtheorem{example}[thm]{Example}
\newtheorem{que}[thm]{Question}
\newtheorem{quest}[thm]{Quest}
\newtheorem{prob}[thm]{Problem}

\setcounter{footnote}{1}

\title{\bf Coexistence and duality in competing species models\footnote{This article was written for the learning session `Competing species models' 
organized by the authors and held in August 2017 at the Institute for Mathematical Sciences at the National University of Singapore. }}
\author{Yu-Ting Chen\footnote{Department of Mathematics, University of Tennessee, Knoxville, TN, United States of America.} and Matthias Hammer\footnote{Institut f\"ur Mathematik, Technische Universit\"at Berlin, Stra{\ss}e des 17. Juni 136, 10623 Berlin, Germany.}}
\date{\today}

\maketitle
\abstract{We discuss some stochastic spatial generalizations of the Lotka--Volterra model for competing species. The generalizations take the forms of spin systems on general discrete sets and interacting diffusions on integer lattices. 
Methods for proving coexistence in these generalizations and some related open questions are discussed. We use duality as the central point of view. It relates coexistence of the models to survival of their dual processes. \\
 
\smallskip 

}

\tableofcontents

\section{Introduction}
Competition of species is a basic phenomenon studied in ecological dynamics. The mathematical analysis goes back to the foundational model by Lotka and Volterra in \cite{Lotka,Volterra}. 
Total sizes of two populations of different species, or types, are the main objects of study. The populations are assumed well-mixed, and their sizes are in the continuum. 
The dynamics for these total sizes follow a two-dimensional extension of the logistic differential equation.
Now restricted resources and competition among individuals in the populations suppress the exponential growths.  
See the monograph by Hofbauer and Sigmund~\cite{HS} for an excellent introduction to this classical model.

In the biological literature, choosing new ingredients to extend the Lotka--Volterra model  has been a subject of research during the last few decades.
One popular method is the use of stochasticity and spatial structure.
By stochasticity,  types of individuals in the whole population  are randomly determined. In the presence of spatial structure, it is no longer the case that everyone can interact with everyone else. Populations are arranged according to Euclidean spaces, integer lattices, or more general geometric structures. In particular, the  Euclidean space and integer lattice in two dimensions are biologically relevant. The corresponding tori are useful for mathematical analysis and computer experiments. 

Spatial structure and stochasticity can be combined to create new properties. 
Spatial structure naturally points to the use of partial differential equations, and in the presence of stochasticity, of stochastic partial differential equations. See \cite{HLBV,TM} for some examples. Besides, stochastic models on discrete spatial structures arise as popular candidates to incorporate both space and stochasticity. They can be connected to stochastic partial differential equations via scaling limits \cite{CP:Rescaled, CMP}, so that useful tools from stochastic analysis can be applied.  
On the other hand, these discrete models can circumvent the serious issue of well-posedness of stochastic partial differential equations beyond one spatial dimension.

This article is an introduction to stochastic spatial generalizations of the Lotka--Volterra model. For definiteness, our discussions are centered around two well-known generalizations, one due to Neuhauser and Pacala \cite{NP} and the other due to Blath, Etheridge and Meredith \cite{BEM}. We explain in Section~\ref{sec:LV} how these models are natural extensions of the Lotka--Volterra model. These extensions incorporate stochasticity in the forms of the voter model and the stepping stone model.

 The models in \cite{NP,BEM} can be briefly described as follows in terms of the spatial structures and the dynamics. 
 First, the model by Neuhauser and Pacala is a dynamical model (started at time $0$) for individuals in two populations of different types. These individuals 
occupy different sites of a discrete space, and so the model is at the microscopic level. (In terms of the terminology in \cite[Chapter~III]{Liggett}, the model is a spin system.)
In contrast, the model by Blath et al.~\cite{BEM} considers continuous population sizes as in the Lotka--Volterra model, but here the two populations are structured into `villages' occupying vertices of $\Z^d$. Given weights for the strengths of interactions among the villages, a countable system of interacting diffusion processes defines the dynamics of the sizes of the populations.

As in the Lotka--Volterra model, the major interest in these generalizations is centered around the question of characterizing large-time behavior. In these lecture notes, we explain the related properties with the use of duality. For the Neuhauser--Pacala model, our discussion considers the associated spin system $(\eta_t)$ in terms of its dual process $(\widehat{\eta}_t)$.  
These two processes are related by equations taking the following form for suitable bivariate dual functions $H$:
\begin{align}\label{analyticdual}
\E[H(\eta_t,\widehat{\eta}_0)]=\E[H(\eta_0,\widehat{\eta}_t)] ,
\end{align}
when $\eta_0$ and $\widehat{\eta}_0$ are both constant. 
Hence, the question of characterizing the large-time behavior of $(\eta_t)$ can be converted to the same question for $(\widehat{\eta}_t)$, and the hope is that the dual process is easier to study. 
In the case of \cite{NP}, duality can even take a finer form, namely Harris's graphical duality, which is sometimes also called pathwise duality. Here, $(\eta_t)$ and $(\widehat{\eta}_t)$ can be coupled on the same probability space. See \cite{Durrett_note,JK13,SS}. For the case of interacting diffusions in \cite{BEM}, we discuss large-time behavior of the model using duality in the form (\ref{analyticdual}). Finally, we refer the reader to the lecture notes in the present series by Sturm, Swart and V\"ollering~\cite{SSV} and to the papers by Jansen and Kurt~\cite{JK13, JK14} for further discussions of duality.

\paragraph{\bf Organization of this article.} This article is organized as follows. In Section~\ref{sec:LV}, we first explain some features of the Lotka--Volterra model that motivate the generalizations in \cite{NP,BEM}. 
Section~\ref{sec:spin} is devoted to a discussion of the (symmetric) Neuhauser--Pacala model. We first explain the duality from various points of view. Then we explain its usefulness for solving the stationarity of the model. 
Section~\ref{sec:interacting} continues the spirit of Section~\ref{sec:spin} for a discussion of the Blath--Etheridge--Meredith model.

\paragraph{\bf Acknowledgements.} Supports from the Institute for Mathematical Sciences at the National University of Singapore for both authors and a grant from the Simons Foundation for Y.-T. C. are gratefully acknowledged.

\section{Death-birth events in the Lotka--Volterra model }\label{sec:LV}
In this section, we discuss the Lotka--Volterra model and its connections with the stochastic spatial generalizations in the main papers under consideration \cite{NP,BEM}.

Recall that the competitive Lotka--Volterra ordinary differential equations for two populations of different types are based on the logistic equations. There are interactions within and between the two populations.
If we denote by $N_0=N_0(t)$ and $N_1=N_1(t)$ the total sizes of the populations of type $0$ and $1$, respectively, then  the system of dynamical equations for $N_0$ and $N_1$ is given by
\begin{align}
\left\{
\begin{array}{ll}
\dot{N_0}=&\!\!\!\!\displaystyle r_0N_0\left(1-\frac{N_0+\alpha_{01}N_1}{K_0}\right),\\
\vspace{-.3cm}\\
\dot{N_1}=&\!\!\!\! \displaystyle r_1N_1\left(1-\frac{N_1+\alpha_{10}N_0}{K_1}\right).
\end{array}
\right.\label{eq:LV0}
\end{align}
Here, $r_i$ is the  intrinsic growth rate of the population of type $i$, $K_i$ is the carrying capacity of the population of type $i$, and $\alpha_{ij}$ is called the strength of interspecific competition. These are given  positive constants.
Without the term $(N_0+\alpha_{01}N_1)/K_0$ to suppress the growth of $0$-individuals, they grow exponentially with rate $r_0$. A similar interpretation applies to the equation for $N_1$. 
By (\ref{eq:LV0}), the total population size $N_0+N_1$ does not stay constant.

Now we rewrite (\ref{eq:LV0}) in the following form (with obvious definitions of constants):
\begin{align}
\left\{
\begin{array}{ll}
\dot{N_0}=&\!\!\!\!\displaystyle 
N_0(a-bN_0-cN_1),\\
\vspace{-.3cm}\\
\dot{N_1}=&\!\!\!\! \displaystyle 
N_1(a'-b'N_1-c'N_0),
\end{array}
\right.\label{eq:LV}
\end{align}
and then turn to the densities of $0$-individuals and $1$-individuals defined by
\[
p_0\stackrel{\rm def}{=}\frac{N_0}{N_0+N_1}\quad\mbox{and}\quad p_1\stackrel{\rm def}{=}\frac{N_1}{N_0+N_1}.
\] 
By (\ref{eq:LV}), we get
\begin{align}
\dot{p_0}
=&\,-\frac{N_0N_1}{(N_0+N_1)^2}\left(a'+bN_0+cN_1\right)+\frac{N_0N_1}{(N_0+N_1)^2}\left(a+b'N_1+c'N_0\right).\label{p0dot}
\end{align}

To simplify the last equation, we assume
\begin{align}\label{a=a'}a=a',
\end{align}
that is, intrinsic growth rates $r_0$ and $r_1$ are  identical.
Then (\ref{p0dot}) reduces to the following equation:
\begin{align}\label{p:de}
\dot{p_0}=-\frac{N_0N_1}{(N_0+N_1)^2}\left(bN_0+cN_1\right)+\frac{N_0N_1}{(N_0+N_1)^2}\left(b'N_1+c'N_0\right).
\end{align}
Now, we change time by using the functions
\begin{align}\label{pj:new}
t\longmapsto  p_j\left(\int_0^t \frac{1}{(\lambda p_1(s)+p_0(s)) b'(N_0+N_1)}ds\right),\quad j=0,1.
\end{align}
Then (\ref{p:de}) can be transformed into a closed equation, since this time-change amounts to the effect of  dividing both sides of (\ref{p:de}) by 
\[
(\lambda p_1+p_0) b'(N_0+N_1). 
\]
To simplify notation, for $j=0,1$, we still denote the corresponding new function in (\ref{pj:new}) by $p_j$. Then (\ref{p:de}) implies
\begin{align}\label{MF}
\dot{p_0}=-p_0\frac{\lambda p_1}{\lambda p_1+p_0}(p_0+\alpha_{01}p_1)+p_1\frac{p_0}{\lambda p_1+p_0}(p_1+\alpha_{10}p_0).
\end{align}
The coefficients of (\ref{p:de}) follow since, under the present assumption in (\ref{a=a'}), the constants $a,b,c,a',b',c'$ in (\ref{eq:LV}) satisfy the following equations: 
\begin{align}\label{def:lambda}
\lambda\stackrel{\rm def}{=}b/b'=K_1/K_0, \quad c/b=\alpha_{01}\quad\mbox{ and }\quad c'/b'=\alpha_{10}.
\end{align}

Equation (\ref{MF}) suggests natural interpretations similar to the Moran process in population genetics. The first term in (\ref{MF}), that is
\begin{align}\label{db}
-p_0\frac{\lambda p_1}{\lambda p_1+p_0}(p_0+\alpha_{01}p_1)=-p_0\times(p_0+\alpha_{01}p_1)\times \frac{\lambda p_1}{\lambda p_1+p_0},
\end{align}
has the interpretation as
\begin{align}
\begin{split}
-\mbox{(probability to find a $0$-individual)}
& \times  \mbox{(death rate of a $0$-individual)}\\
& \times\mbox{(birth probability of a $1$-individual)}.
\end{split}
\end{align}
Here, taking
\[
p_0+\alpha_{01}p_1
\]
as a death rate is consistent with the dynamical equation of $N_0$ in (\ref{eq:LV0}) now that $r_0=r_1$ by our assumption in (\ref{a=a'}). Indeed, $p_0+\alpha_{01}p_1$ is the same as $N_0+\alpha_{01}N_1$ up to the (time-dependent) total population size $N_0+N_1$. 
Also, by the definition of $\lambda$ in (\ref{def:lambda}),
\[
\frac{\lambda p_1}{\lambda p_1+p_0}=\frac{K_1N_1}{K_1N_1+K_0N_0}.
\]
This ratio can be taken as a birth probability of $1$-individuals. The higher the carrying capacity $K_1$ is, the more likely $1$-individuals can replicate. These altogether are responsible for the decrease of $p_0$, which explains the minus sign in (\ref{db}).
The second term in (\ref{MF}) can be interpreted similarly with the roles of $1$ and $0$ exchanged. But now this term is responsible for the increase of $0$-individuals in the population. 

Equilibrium in 
the dynamical equation for $p_0$ in (\ref{MF}) is very easy to obtain. Since $p_0+p_1=1$, we get
\begin{align}\label{de:p0}
\begin{split}
\dot{p_0}&=\frac{p_0(1-p_0)}{\lambda (1-p_0)+p_0}\big\{(1-\lambda \alpha_{01})-p_0\big[(1-\lambda \alpha_{01})+(\lambda-\alpha_{10})\big]\big\}\\
&=\frac{F(p_0)}{\lambda(1-p_0)+p_0},
\end{split}
\end{align}
where $F$ is a polynomial in $p_0$. 
Setting the right-hand side to zero shows that the foregoing equation has a unique stable equilibrium given by
\begin{align}\label{p0*}
p_0^*=\frac{(1-\lambda \alpha_{01})}{(1-\lambda \alpha_{01})+(\lambda-\alpha_{10})}
\end{align}
if $0\leq \alpha_{10}<\lambda$ and $0\leq \alpha_{01}<1/\lambda$, where the restrictions on $\alpha_{10}$ and $\alpha_{01}$ ensure that $p_0^*$ falls in $(0,1)$. The stability of $p_0^*$ is plain from the graph of $F$ given in Figure~\ref{Fig1}.

\begin{figure}[t]
\centering
\vspace{-2cm}
 \hspace{2cm} 
\setlength{\unitlength}{0.5cm}
\begin{picture}(5,10)(6,-6)
\put(3,-3){\vector(1,0){7}}
\put(4,-4){\vector(0,1){4}}
\put(10,-4){$p$}
\put(2,-1){$F(p)$}
\thicklines
\put(8,-3){\vector(-1,0){2}}
\put(4,-3){\vector(1,0){2}}
\qbezier(4,-3)(5,-1)(6,-3)
\qbezier(6,-3)(7,-5)(8,-3)
\put(4,-4){$0$}
\put(5.5,-4){$p_0^*$}
\put(8,-4){$1$}
\end{picture}
\vspace{-.5cm}
  \caption{}\label{Fig1}
\end{figure}

A stochastic spatial generalization of (\ref{de:p0}) to spin systems is introduced by Neuhauser and Pacala in \cite{NP}. Put in a general framework, the model can be described in the following way. 
Let $q$ be an irreducible  transition probability kernel on a nonempty set $E$ and suppose that $q$ has a zero trace: 
\begin{align*}
\sum_{x\in E}q(x,x)=0. 
\end{align*}
For  $\sigma\in \{0,1\}$, $x\in E$ and $\eta\in \{0,1\}^E$, we define the local frequencies of $\sigma$'s by
\begin{align*}
f_\sigma (x,\eta)=\sum_{y\in E}q(x,y)\1_{\{\sigma\}}\big(\eta(y)\big).
\end{align*}
The canonical example here is that $q$ is given by the transition probability kernel of simple random walk on a connected graph. 
In this case, $f_\sigma(x,\eta)$ reduces to the usual frequency of $\sigma$'s in the neighborhood of $x$: 
\[
f_\sigma(x,\eta)=\frac{1}{\deg(x)}\sum_{y:y\sim x}\1_{\{\sigma\}}\big(\eta(y)\big),
\]
where $\deg(x)$ denotes the number of neighbors of $x$ and $y\sim x$ means that $y$ is a neighbor of $x$.
The use of these kernels $(E,q)$ is meant to allow perturbations of graph structures in the usual $\Z^2$ or two-dimensional discrete tori. 
Hence, we can view the spin system to be defined below from a more general perspective.

Now the flip rates of the sought-after generalization are set to be 
\begin{align}
\begin{split}\label{def:flip rate}
&0\to 1\mbox{ with rate }(f_0+\alpha_{01}f_1)
\left(\frac{\lambda f_1}{\lambda f_1+f_0}\right),\\
&1\to 0\mbox{ with rate }(f_1+\alpha_{10}f_0) \left(\frac{f_0}{\lambda f_1+f_0}\right).
\end{split}
\end{align} 
In this spin system,
\[
f_0+\alpha_{01}f_1\quad\mbox{ and }\quad f_1+\alpha_{10}f_0
\]
are the death rates and
\[
\frac{\lambda f_1}{\lambda f_1+f_0}\quad\mbox{and}\quad \frac{f_0}{\lambda f_1+f_0}
\] 
are the birth probabilities, as in (\ref{db}). Hence, in the sense of flip rates in \cite[page 122--123]{Liggett}, the $\{0,1\}^E$-valued Markov process $(\eta_t)$ under consideration is characterized by the following limits:
\begin{align}
\begin{split}\label{def:flip}
&\P(\eta_t(x)=1|\eta_0(x)=0)=[f_0(x,\eta_0)+\alpha_{01}f_1(x,\eta_0)]\left(\frac{\lambda f_1(x,\eta_0)}{\lambda f_1(x,\eta_0)+f_0(x,\eta_0)}\right)+o(t),\quad t\searrow 0+,\\
&\P(\eta_t(x)=0|\eta_0(x)=1)=[f_1(x,\eta_0)+\alpha_{10}f_0(x,\eta_0)]\left(\frac{ f_0(x,\eta_0)}{\lambda f_1(x,\eta_0)+f_0(x,\eta_0)}\right)+o(t),\quad t\searrow 0+,\\
&\P\big(\eta_t(x)\neq \eta_0(x),\eta_t(y)\neq \eta_0(y)\big)=o(t),\quad t\searrow 0+,\quad \forall\;x\neq y.
\end{split}
\end{align}

The flip rates  introduced in (\ref{def:flip rate}) are applicable for general $(E,q)$'s defined above. In particular, it can be shown by limit theorems of semimartingales 
(cf. \cite{JS})
that the density of $1$'s in this particle system on a complete graph over $N$ vertices converges to the solution of (\ref{MF})  as $N\to\infty$ in the space of c\`adl\`ag functions equipped with Skorokhod's $J_1$-topology.

Our discussions in the rest of this paper for this model  will be on the construction of this spin system in the symmetric case with $\alpha_{01}=\alpha_{10}$ and $\lambda=1$
 and on the analysis of its equilibrium. Note that if $\alpha_{01}=\alpha_{10}=1$ and $\lambda=1$, then the model reduces to the voter model and is henceforth excluded from the discussion below.

A model introduced by Blath, Etheridge and Meredith~\cite{BEM} generalizes the Lotka--Volterra differential equations in (\ref{eq:LV0}) from a  point of view very similar to those 
in the biology papers
by Bolker and Pacala~\cite{BP} and Murrell and Law \cite{ML}. Here, we only consider the case that the underlying spatial structure is  an integer lattice $\Z^d$ for some $d\geq 1$. Population sizes at points in $\Z^d$ are in the continuum and subject to the new features of migration and stochastic growth defined by branching noises. More precisely, $(\eta_t)=(\eta_t(x);x\in \Z^d)$ and 
$(\eta'_t)=(\eta'_t(x);x\in \Z^d)$ model population sizes of $0$-individuals and $1$-individuals at all sites $x\in \Z^d$, respectively, are given by the following system of stochastic differential equations: for all $x\in \Z^d$,
\begin{align}
\begin{split}\label{sde:xi}
d\eta_t(x)=&\sum_{y\in \Z^d}m_{xy}\big(\eta_t(y)-\eta_t(x)\big)dt\\
&+\eta_t(x)\left(\alpha-\sum_{y\in \Z^d}\beta_{xy}\eta_t(y)-\sum_{y\in \Z^d}\gamma_{xy}\eta'_t(y)\right)dt+\sqrt{\sigma \eta_t(x)}dB_t(x),
\end{split}\\
\begin{split}\label{sde:xi'}
d\eta'_t(x)=&\sum_{y\in \Z^d}m'_{xy}\big(\eta'_t(y)-\eta'_t(x)\big)dt\\
&+\eta'_t(x)\left(\alpha'-\sum_{y\in \Z^d}\beta'_{xy}\eta'_t(y)-\sum_{y\in \Z^d}\gamma'_{xy}\eta_t(y)\right)dt+\sqrt{\sigma \eta'_t(x)}dB'_t(x),
\end{split}
\end{align}
where $\{B(x),B'(x);x\in \Z^d\}$ are i.i.d. one-dimensional standard Brownian motions.

The equations in (\ref{sde:xi}) and (\ref{sde:xi'}) now model spatial effects in a different way. For example, the sum 
\[
\sum_{y\in \Z^d}m_{xy}\big(\eta_t(y)-\eta_t(x)\big)dt
\]
models migration of individuals by a nonnegative matrix $m$ since 
\[
\mathsf Lf(x)= \sum_{y\in \Z^d}m_{xy}[f(y)-f(x)]
\]
is the generator of a Markov chain  moving along sites of $\Z^d$ subject to the $q$-matrix $\{m_{xy};x\neq y\}$. 
Also, the term
\[
\eta_t(x)\left(\alpha-\sum_{y\in \Z^d}\beta_{xy}\eta_t(y)-\sum_{y\in \Z^d}\gamma_{xy}\eta'_t(y)\right)dt
\]
has a natural correspondence to the logistic differential equations in (\ref{eq:LV0}). It is assumed in addition that these matrices $m,m',\beta,\beta',\gamma,\gamma'$ 
are `homogeneous' and
have finite ranges in the sense that, for example, $m_{xy}$ depends only on $\|x-y\|_\infty$ and
 $m_{xy}=0$ for all $\|x-y\|_\infty\geq L$ for some $L>0$ independent of $x,y$.  
 
The objects corresponding to (\ref{MF}) are the density processes 
\[
p_t(x)\stackrel{\rm def}{=}\frac{\eta_t(x)}{\eta_t(x)+\eta'_t(x)}
\]
of $0$-individuals at all sites $x$. The derivation in (\ref{p0dot}) can be generalized by It\^{o}'s formula
and formally conditioning on $\eta_t(x)+\eta'_t(x)\equiv N$ for all $x$. 
Then, assuming also that $m_{xy}=m'_{xy}$ and that $\beta_{xy},\beta'_{xy},\gamma_{xy},\gamma'_{xy}$ are zero for $x\neq y$ (purely local interactions), one is led to the following system:
\begin{align}\label{p:SDE}
\begin{split}
dp_t(x)&=\sum_{y\in \Z^d}m_{xy}\big(p_t(y)-p_t(x)\big)dt + s p_t(x)\big(1-p_t(x)\big)\big(1-\mu p_t(x)\big)dt\\
&\quad + \sqrt{N^{-1}p_t(x)\big(1-p_t(x)\big)}dW_t(x),\quad x\in\mathbb Z^d,
\end{split}
\end{align}
where $\{W(x); x\in\Z^d\}$ is a system of independent standard Brownian motions, and $s$ and $\mu$ are real parameters which can be expressed explicitly in terms of
the parameters of $\{\eta_t\}$ and $\{\eta'_t\}$.
We refer the reader to \cite[pages~1482--1483]{BEM} for details of the derivation. 
In particular, if we consider the symmetric case where $\eta$ and $\eta'$ follow the same parameters, then $\mu=2$ and $s = \beta_{xx} -\gamma_{xx}$.
Note that existence and uniqueness (in the strong sense) of a $[0,1]^{\Z^d}$-valued solution
of the system of SDEs in (\ref{p:SDE}) can be obtained independently from classical results for infinite-dimensional SDEs due to Shiga and Shimizu~\cite[Theorem~2.1]{SS_1980}. 
In Section~\ref{sec:interacting}, the discussion for these interacting diffusions in (\ref{p:SDE}) is independent of the system defined by (\ref{sde:xi}), and so we use general parameters $s$ and $\mu$ from now on.

\section{Spin systems on discrete sets}\label{sec:spin}
Throughout this section, we focus on the symmetric Neuhauser--Pacala model, that is $\lambda=1$ and 
\begin{linenomath*}\begin{align}\label{def:alpha}
\alpha_{01}=\alpha_{10}=\alpha\in [0,1).
\end{align}\end{linenomath*}
In this case, there is no bias in birth rates and death rates induced by these parameters, and  the flip rates defined in (\ref{def:flip rate}) can be simplified as follows:
\begin{linenomath*}\begin{align}\label{def:fr1}
\left\{
\begin{array}{ll}
0\to 1&\mbox{ with rate }(f_0+\alpha f_1)f_1=(1-\alpha)f_0f_1+\alpha f_1,\\
\vspace{-.4cm}\\
1\to 0&\mbox{ with rate }(f_1+\alpha f_0)f_0=(1-\alpha)f_1f_0+\alpha f_0.
\end{array}
\right.
\end{align}\end{linenomath*}
The assumptions that $\lambda=1$ and $\alpha_{01}=\alpha_{10}$ mean that the populations of types $0$ and $1$ behave in the same way.

Below we first introduce in Section~\ref{sec:NP_duality} various constructions of this symmetric model. The constructions will progressively lead us to the duality for the model, which is called {\bf parity duality} in this paper for reasons that shall become self-evident. 
In Section~\ref{sec:NP_inv}, we discuss the related basic results and open questions.

\subsection{Constructions and parity duality}\label{sec:NP_duality}
In the following, we first view the symmetric Neuhauser--Pacala model in terms of two sets of flip rates. These sets of flip rates define the symmetric model with a detailed coupling by independent Poisson processes. On a finite set, the construction leads to a matrix representation of the model and easily induces the parity duality. We will explain these steps in detail. By the end of this subsection, we give a sketch of how generalizations of the parity duality  can be obtained on infinite sets.   

First,  the flip rates defined above in (\ref{def:fr1}) are  decomposed into  two fundamental mechanisms given by
\begin{linenomath*}\begin{align}\label{def:fr11}
\begin{split}
&\mbox{pairwise annihilation:}\left\{
\begin{array}{ll}
0\to 1&\mbox{ with rate }(1-\alpha)f_0f_1,\\
\vspace{-.4cm}\\
1\to 0&\mbox{ with rate }(1-\alpha)f_1f_0,
\end{array}
\right.\\
&\hspace{2.46cm}\mbox{voting:}\left\{
\begin{array}{ll}
0\to 1&\mbox{ with rate }\alpha f_1,\\
\vspace{-.4cm}\\
1\to 0&\mbox{ with rate }\alpha f_0.
\end{array}
\right.
\end{split}
\end{align}\end{linenomath*}
As before, we suppress sites and configurations to be updated in this notation. The corresponding $\{0,1\}^E$-valued Markov processes can be characterized analogously as in (\ref{def:flip}).  

The first set of flip rates in (\ref{def:fr11}) corresponds to pairwise annihilation. The children of two neighbors randomly chosen according to the transition probability kernel $q$ try to invade a  focal site, namely the site chosen to be updated, subject to pairwise annihilation. To see this interpretation, first we write the flip rates for site $x$ given population configuration $\eta$ as follows:
\begin{linenomath*}\begin{align*}
(1-\alpha)f_0(x,\eta)f_1(x,\eta)&=(1-\alpha)f_1(x,\eta)f_0(x,\eta)\\
&=(1-\alpha)\left(\sum_{y\in E}q(x,y)\eta(y)\right)\left(\sum_{z\in E}q(x,z)[1-\eta(z)]\right).
\end{align*}\end{linenomath*}
Hence, a spin flip at $x$ is triggered if and only if a $1$-individual $y$ and a $0$-individual $z$ are chosen (according to the kernel $q$).

Let us use this property to explain the implied effect of pairwise annihilation.
If the underlying spatial structure is a graph and we let a filled vertex denote a vertex occupied by a $1$-individual and an empty vertex denote one occupied by a $0$-individual, then we can visualize the complete set of possible transitions as follows:
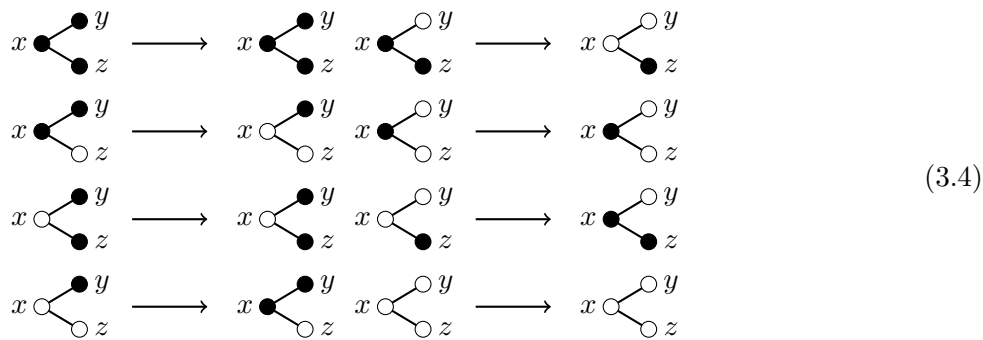
\begin{linenomath*}\begin{align}\label{diag_2}
\begin{split}
\begin{tikzpicture}
\node at (-.3,.3) {$x$};
\node at (.8,.6) {$y$};
\node at (.8,0) {$z$};
\node at (2.7,.3) {$x$};
\node at (3.8,.6) {$y$};
\node at (3.8,.0) {$z$};
\draw[thick] (0,.3) -- (.5,0); 
\draw[thick](0,.3)--(.5,.6);
\draw[thick] (3,.3) -- (3.5,0); 
\draw[thick](3,.3)--(3.5,.6);
\draw[thick,->] (1.2,.3)--(2.2,.3);
\draw[fill=black] (0,.3) circle (3pt);
\draw[fill=black] (.5,0) circle (3pt);
\draw[fill=black] (.5,.6) circle (3pt);
\draw[fill=black] (3,.3) circle (3pt);
\draw[fill=black] (3.5,0) circle (3pt);
\draw[fill=black] (3.5,.6) circle (3pt);
\end{tikzpicture}
\begin{tikzpicture}
\node at (-.3,.3) {$x$};
\node at (.8,.6) {$y$};
\node at (.8,0) {$z$};
\node at (2.7,.3) {$x$};
\node at (3.8,.6) {$y$};
\node at (3.8,.0) {$z$};
\draw[thick] (0,.3) -- (.5,0); 
\draw[thick](0,.3)--(.5,.6);
\draw[thick] (3,.3) -- (3.5,0); 
\draw[thick](3,.3)--(3.5,.6);
\draw[thick,->] (1.2,.3)--(2.2,.3);
\draw[fill=black] (0,.3) circle (3pt);
\draw[fill=black] (.5,0) circle (3pt);
\draw[fill=white] (.5,.6) circle (3pt);
\draw[fill=white] (3,.3) circle (3pt);
\draw[fill=black] (3.5,0) circle (3pt);
\draw[fill=white] (3.5,.6) circle (3pt);
\end{tikzpicture}\\[-5pt]
\begin{tikzpicture}
\node at (-.3,.3) {$x$};
\node at (.8,.6) {$y$};
\node at (.8,0) {$z$};
\node at (2.7,.3) {$x$};
\node at (3.8,.6) {$y$};
\node at (3.8,.0) {$z$};
\draw[thick] (0,.3) -- (.5,0); 
\draw[thick](0,.3)--(.5,.6);
\draw[thick] (3,.3) -- (3.5,0); 
\draw[thick](3,.3)--(3.5,.6);
\draw[thick,->] (1.2,.3)--(2.2,.3);
\draw[fill=black] (0,.3) circle (3pt);
\draw[fill=white] (.5,0) circle (3pt);
\draw[fill=black] (.5,.6) circle (3pt);
\draw[fill=white] (3,.3) circle (3pt);
\draw[fill=white] (3.5,0) circle (3pt);
\draw[fill=black] (3.5,.6) circle (3pt);
\end{tikzpicture}
\begin{tikzpicture}
\node at (-.3,.3) {$x$};
\node at (.8,.6) {$y$};
\node at (.8,0) {$z$};
\node at (2.7,.3) {$x$};
\node at (3.8,.6) {$y$};
\node at (3.8,.0) {$z$};
\draw[thick] (0,.3) -- (.5,0); 
\draw[thick](0,.3)--(.5,.6);
\draw[thick] (3,.3) -- (3.5,0); 
\draw[thick](3,.3)--(3.5,.6);
\draw[thick,->] (1.2,.3)--(2.2,.3);
\draw[fill=black] (0,.3) circle (3pt);
\draw[fill=white] (.5,0) circle (3pt);
\draw[fill=white] (.5,.6) circle (3pt);
\draw[fill=black] (3,.3) circle (3pt);
\draw[fill=white] (3.5,0) circle (3pt);
\draw[fill=white] (3.5,.6) circle (3pt);
\end{tikzpicture}\\[-5pt]
\begin{tikzpicture}
\node at (-.3,.3) {$x$};
\node at (.8,.6) {$y$};
\node at (.8,0) {$z$};
\node at (2.7,.3) {$x$};
\node at (3.8,.6) {$y$};
\node at (3.8,.0) {$z$};
\draw[thick] (0,.3) -- (.5,0); 
\draw[thick](0,.3)--(.5,.6);
\draw[thick] (3,.3) -- (3.5,0); 
\draw[thick](3,.3)--(3.5,.6);
\draw[thick,->] (1.2,.3)--(2.2,.3);
\draw[fill=white] (0,.3) circle (3pt);
\draw[fill=black] (.5,0) circle (3pt);
\draw[fill=black] (.5,.6) circle (3pt);
\draw[fill=white] (3,.3) circle (3pt);
\draw[fill=black] (3.5,0) circle (3pt);
\draw[fill=black] (3.5,.6) circle (3pt);
\end{tikzpicture}
\begin{tikzpicture}
\node at (-.3,.3) {$x$};
\node at (.8,.6) {$y$};
\node at (.8,0) {$z$};
\node at (2.7,.3) {$x$};
\node at (3.8,.6) {$y$};
\node at (3.8,.0) {$z$};
\draw[thick] (0,.3) -- (.5,0); 
\draw[thick](0,.3)--(.5,.6);
\draw[thick] (3,.3) -- (3.5,0); 
\draw[thick](3,.3)--(3.5,.6);
\draw[thick,->] (1.2,.3)--(2.2,.3);
\draw[fill=white] (0,.3) circle (3pt);
\draw[fill=black] (.5,0) circle (3pt);
\draw[fill=white] (.5,.6) circle (3pt);
\draw[fill=black] (3,.3) circle (3pt);
\draw[fill=black] (3.5,0) circle (3pt);
\draw[fill=white] (3.5,.6) circle (3pt);
\end{tikzpicture}\\[-5pt]
\begin{tikzpicture}
\node at (-.3,.3) {$x$};
\node at (.8,.6) {$y$};
\node at (.8,0) {$z$};
\node at (2.7,.3) {$x$};
\node at (3.8,.6) {$y$};
\node at (3.8,.0) {$z$};
\draw[thick] (0,.3) -- (.5,0); 
\draw[thick](0,.3)--(.5,.6);
\draw[thick] (3,.3) -- (3.5,0); 
\draw[thick](3,.3)--(3.5,.6);
\draw[thick,->] (1.2,.3)--(2.2,.3);
\draw[fill=white] (0,.3) circle (3pt);
\draw[fill=white] (.5,0) circle (3pt);
\draw[fill=black] (.5,.6) circle (3pt);
\draw[fill=black] (3,.3) circle (3pt);
\draw[fill=white] (3.5,0) circle (3pt);
\draw[fill=black] (3.5,.6) circle (3pt);
\end{tikzpicture}
\begin{tikzpicture}
\node at (-.3,.3) {$x$};
\node at (.8,.6) {$y$};
\node at (.8,0) {$z$};
\node at (2.7,.3) {$x$};
\node at (3.8,.6) {$y$};
\node at (3.8,.0) {$z$};
\draw[thick] (0,.3) -- (.5,0); 
\draw[thick](0,.3)--(.5,.6);
\draw[thick] (3,.3) -- (3.5,0); 
\draw[thick](3,.3)--(3.5,.6);
\draw[thick,->] (1.2,.3)--(2.2,.3);
\draw[fill=white] (0,.3) circle (3pt);
\draw[fill=white] (.5,0) circle (3pt);
\draw[fill=white] (.5,.6) circle (3pt);
\draw[fill=white] (3,.3) circle (3pt);
\draw[fill=white] (3.5,0) circle (3pt);
\draw[fill=white] (3.5,.6) circle (3pt);
\end{tikzpicture}
\end{split}
\end{align}\end{linenomath*}
Consider the following two examples from the first row of (\ref{diag_2}).
For the transition on the right-hand side, the focal site is $x$, and the two neighbors at sites  $y$ and $z$, randomly chosen according to the probability $q(x,\cdot)$ without replacement, try to invade site $x$ by their children. The resident $1$-individual at $x$ and the invading child of the $1$-individual at $z$ annihilate each other, and so the child of the $0$-individual at $y$ takes over the site $x$ after the update. A similar interpretation applies to the focal site $x$ for the transition on the left. But this time, two of the three $1$-individuals (two from $y$ and $z$ plus the resident $1$-individual) annihilate each other, so that there is only one $1$-individual left at $x$ (we do not care where this survivor comes from). 

The second set of flip rates in (\ref{def:fr11}) defines a voting mechanism. Now, the flip rates can be written as
\[
\alpha f_1(x,\eta)=\alpha\sum_{y\in E}q(x,y)\eta(y)\quad\mbox{and}\quad \alpha f_0(x,\eta)=\alpha\sum_{y\in E}q(x,y)[1-\eta(y)]
\]
so that there is a spin flip at $x$ if and only if a neighbor of the opposite type is chosen:
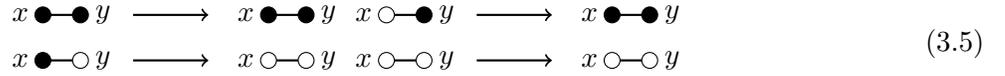
\begin{linenomath*}\begin{align}
\begin{split}\label{diag_1}
\begin{tikzpicture}
\node at (-.3,.3) {$x$};
\node at (.8,0.3) {$y$};
\node at (3.8,.3) {$y$};
\node at (2.7,.3) {$x$};
\draw[thick](0,.3)--(.5,.3);
\draw[thick](3,.3)--(3.5,.3);
\draw[thick,->] (1.2,.3)--(2.2,.3);
\draw[fill=black] (0,.3) circle (3pt);
\draw[fill=black] (.5,.3) circle (3pt);
\draw[fill=black] (3,.3) circle (3pt);
\draw[fill=black] (3.5,.3) circle (3pt);
\end{tikzpicture}
\begin{tikzpicture}
\node at (-.3,.3) {$x$};
\node at (.8,0.3) {$y$};
\node at (3.8,.3) {$y$};
\node at (2.7,.3) {$x$};
\draw[thick](0,.3)--(.5,.3);
\draw[thick](3,.3)--(3.5,.3);
\draw[thick,->] (1.2,.3)--(2.2,.3);
\draw[fill=white] (0,.3) circle (3pt);
\draw[fill=black] (.5,.3) circle (3pt);
\draw[fill=black] (3,.3) circle (3pt);
\draw[fill=black] (3.5,.3) circle (3pt);
\end{tikzpicture}
\\[-5pt]
\begin{tikzpicture}
\node at (-.3,.3) {$x$};
\node at (.8,0.3) {$y$};
\node at (3.8,.3) {$y$};
\node at (2.7,.3) {$x$};
\draw[thick](0,.3)--(.5,.3);
\draw[thick](3,.3)--(3.5,.3);
\draw[thick,->] (1.2,.3)--(2.2,.3);
\draw[fill=black] (0,.3) circle (3pt);
\draw[fill=white] (.5,.3) circle (3pt);
\draw[fill=white] (3,.3) circle (3pt);
\draw[fill=white] (3.5,.3) circle (3pt);
\end{tikzpicture}
\begin{tikzpicture}
\node at (-.3,.3) {$x$};
\node at (.8,0.3) {$y$};
\node at (3.8,.3) {$y$};
\node at (2.7,.3) {$x$};
\draw[thick](0,.3)--(.5,.3);
\draw[thick](3,.3)--(3.5,.3);
\draw[thick,->] (1.2,.3)--(2.2,.3);
\draw[fill=white] (0,.3) circle (3pt);
\draw[fill=white] (.5,.3) circle (3pt);
\draw[fill=white] (3,.3) circle (3pt);
\draw[fill=white] (3.5,.3) circle (3pt);
\end{tikzpicture}
\end{split}
\end{align}\end{linenomath*}
On the left of the second row of (\ref{diag_1}), the focal site chosen for update is again $x$, and the new type is chosen randomly from one of the neighbors, which is $y$, so that a $0$-individual replaces the resident $1$-individual at $x$.

Now we explain how to use the quotient group $\Z_2=\Z/2\Z$ to `linearize' the mechanisms in (\ref{def:fr11}).
For any $x,y,z\in E$, we first define an $E\times E$ matrix with entries in $\mathbb Z_2$ by
\[\{x\}\times \{y,z\} \stackrel{\rm def}{=} \left(a_{v,w}\right)_{v,w\in E}\quad\text{with non-zero entries } a_{x,y}=a_{x,z}=1.\] 
Then for the pairwise annihilation mechanism, we consider the standard superposition of Poisson processes and the following Poisson updates by linear transformations: for all $\eta\in \mathbb Z_2^E$ and $x,y,z\in E$:
\begin{align}\label{ann}
&\eta\to \eta+\{x\}\times \{y,z\}\eta\quad \mbox{ with rate }r(\{x\}\times \{y,z\})=(1-\alpha) q(x,y)q(x,z).
\end{align}
Here, the configuration after the update can be written as
\[
\eta+\{x\}\times \{y,z\}\eta=[\eta-\eta(x)\1_x]+[\eta(x)+\eta(y)+\eta(z)]\1_x.
\]
The difference in the first pair of brackets shows that the types outside $x$ are kept fixed and we clear the site $x$ by setting it to be zero. The sum in the second pair of brackets  shows precisely pairwise annihilation due to the group structure of $\mathbb Z_2$ and always gives us a value in $\{0,1\}$. Here and in what follows, we continue to use addition modulo $2$ whenever $\Z_2$ is used.

For the voting mechanism, we introduce independent Poisson updates by the following linear transformations:
\begin{align}
\eta\to \eta+\{x\}\times \{x,y\}\eta\quad\mbox{with rate }r(\{x\}\times \{x,y\})=\alpha q(x,y).\label{voter}
\end{align}
Now the configuration after the update can be written as
\begin{align}
\{x\}\times \{x,y\}\eta&=[\eta-\eta(x)\1_x]+[\eta(x)+\eta(x)+\eta(y)]\1_x\notag\\
&=[\eta-\eta(x)\1_x]+\eta(y)\1_x,\label{voter-0}
\end{align}
and so we have removal and adoption of types at $x$ by the first and second terms of (\ref{voter-0}), respectively. We write the matrices defining the updates in (\ref{ann}) and (\ref{voter}) as $\Id+J$ in the following.

On a finite set $E$, the above linear construction of the symmetric Neuhauser--Pacala model by Poisson events can be made more explicit by a random-walk type construction.
Here, we time-change a discrete-time Markov chain by an independent Poisson process (see \cite[Section~2.6]{Norris}). In detail, with initial condition $\1_A$, we define a $\Z_2^E$-valued process by
\begin{linenomath*}\begin{align}\label{construct:linear}
\eta^A_t\stackrel{\rm def}{=} &\;(\Id+J_{N_t})\cdots (\Id+J_1)\1_A, 
\end{align}\end{linenomath*}
where $J_1,J_2,\ldots$ are i.i.d. random matrices with entries in $\{0,1\}$ and law 
\[
\P(J_1=J)=\frac{r(J)}{\sum_{J'}r(J')}
\]
and $(N_t)$ is an independent Poisson process with rate 
\[
\E N_1=\sum_{J'}r(J'). 
\] 

By the construction in (\ref{construct:linear}), we can imitate the usual time-reversal duality of random walks and easily get a duality of the symmetric Neuhauser--Pacala model:
\begin{linenomath*}\begin{align}
\langle \1_B,\eta^A_t\rangle = &\,\1_B^\top (\Id+J_{N_t})\cdots (\Id+J_1)\1_A\label{dual0}\\
=&\,[(\Id+J^\top _{1})\cdots (\Id+J^\top_{N_t})\1_B]^\top \1_A\notag \\
=&\, \langle \widehat{\eta}^{t,B}_t,\1_A\rangle
\stackrel{\rm (d)}{=}\, \langle \widehat{\eta}^B_t,\1_A\rangle,\label{dual}
\end{align}\end{linenomath*}
where $\langle a,b\rangle=\sum_{x\in E}a_xb_x$ for any $a,b\in \Z_2^E$.
In other words, the inner product $(\xi,\eta)\mapsto  \langle \xi,\eta\rangle$  is used as the bivariate dual function and 
\begin{linenomath*}\begin{align}\label{dual+1}
\widehat{\eta}^B_t\stackrel{\rm def}{=}(\Id+J^\top_{N_t})\cdots (\Id+J^\top_{1})\1_B
\end{align}\end{linenomath*}
is the dual Markov chain. 

To specify the update rule of the dual chain defined in (\ref{dual+1}), it is now more convenient to say that there are only $1$-individuals in the population and the sites not occupied by $1$-individuals are left empty. (Recall that for the symmetric Neuhauser--Pacala model, we use the interpretation that there are $0$-individuals and $1$-individuals in the population.)

By transposing the linear transformation in (\ref{ann}), we obtain a mechanism featuring branching with pairwise annihilation:
\begin{linenomath*}\begin{align}\label{ann-2}
&\xi\to \xi+\{y,z\}\times \{x\}\xi\equiv [\xi-\xi(y)\1_y-\xi(z)\1_z]+[\xi(x)+\xi(y)]\1_y+[\xi(x)+\xi(z)]\1_z.
\end{align}\end{linenomath*}
If $x$ is chosen as the focal site and is occupied by a $1$-individual, it gives birth to two children and they try to invade two sites $y$ and $z$ subject to the pairwise annihilation  by the last two terms as before due to the group structure of $\Z_2$. The set of all possible transitions can be visualized as follows:
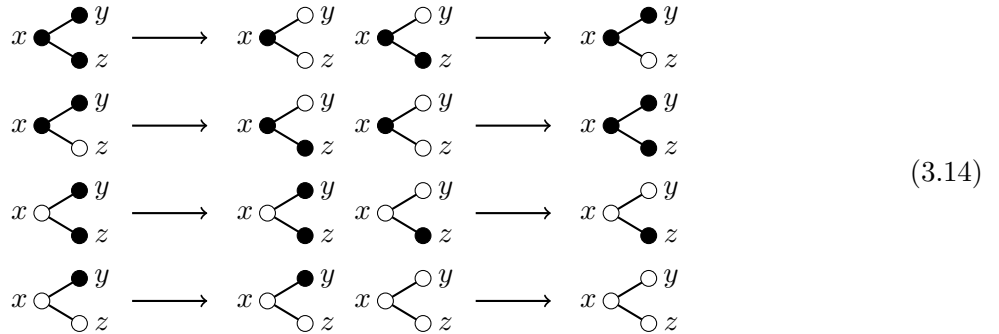
\begin{linenomath*}\begin{align}
\begin{split}\label{diag_dual2}
\begin{tikzpicture}
\node at (-.3,.3) {$x$};
\node at (.8,.6) {$y$};
\node at (.8,0) {$z$};
\node at (2.7,.3) {$x$};
\node at (3.8,.6) {$y$};
\node at (3.8,.0) {$z$};
\draw[thick] (0,.3) -- (.5,0); 
\draw[thick](0,.3)--(.5,.6);
\draw[thick] (3,.3) -- (3.5,0); 
\draw[thick](3,.3)--(3.5,.6);
\draw[thick,->] (1.2,.3)--(2.2,.3);
\draw[fill=black] (0,.3) circle (3pt);
\draw[fill=black] (.5,0) circle (3pt);
\draw[fill=black] (.5,.6) circle (3pt);
\draw[fill=black] (3,.3) circle (3pt);
\draw[fill=white] (3.5,0) circle (3pt);
\draw[fill=white] (3.5,.6) circle (3pt);
\end{tikzpicture}
\begin{tikzpicture}
\node at (-.3,.3) {$x$};
\node at (.8,.6) {$y$};
\node at (.8,0) {$z$};
\node at (2.7,.3) {$x$};
\node at (3.8,.6) {$y$};
\node at (3.8,.0) {$z$};
\draw[thick] (0,.3) -- (.5,0); 
\draw[thick](0,.3)--(.5,.6);
\draw[thick] (3,.3) -- (3.5,0); 
\draw[thick](3,.3)--(3.5,.6);
\draw[thick,->] (1.2,.3)--(2.2,.3);
\draw[fill=black] (0,.3) circle (3pt);
\draw[fill=black] (.5,0) circle (3pt);
\draw[fill=white] (.5,.6) circle (3pt);
\draw[fill=black] (3,.3) circle (3pt);
\draw[fill=white] (3.5,0) circle (3pt);
\draw[fill=black] (3.5,.6) circle (3pt);
\end{tikzpicture}
\\[-5pt]
\begin{tikzpicture}
\node at (-.3,.3) {$x$};
\node at (.8,.6) {$y$};
\node at (.8,0) {$z$};
\node at (2.7,.3) {$x$};
\node at (3.8,.6) {$y$};
\node at (3.8,.0) {$z$};
\draw[thick] (0,.3) -- (.5,0); 
\draw[thick](0,.3)--(.5,.6);
\draw[thick] (3,.3) -- (3.5,0); 
\draw[thick](3,.3)--(3.5,.6);
\draw[thick,->] (1.2,.3)--(2.2,.3);
\draw[fill=black] (0,.3) circle (3pt);
\draw[fill=white] (.5,0) circle (3pt);
\draw[fill=black] (.5,.6) circle (3pt);
\draw[fill=black] (3,.3) circle (3pt);
\draw[fill=black] (3.5,0) circle (3pt);
\draw[fill=white] (3.5,.6) circle (3pt);
\end{tikzpicture}
\begin{tikzpicture}
\node at (-.3,.3) {$x$};
\node at (.8,.6) {$y$};
\node at (.8,0) {$z$};
\node at (2.7,.3) {$x$};
\node at (3.8,.6) {$y$};
\node at (3.8,.0) {$z$};
\draw[thick] (0,.3) -- (.5,0); 
\draw[thick](0,.3)--(.5,.6);
\draw[thick] (3,.3) -- (3.5,0); 
\draw[thick](3,.3)--(3.5,.6);
\draw[thick,->] (1.2,.3)--(2.2,.3);
\draw[fill=black] (0,.3) circle (3pt);
\draw[fill=white] (.5,0) circle (3pt);
\draw[fill=white] (.5,.6) circle (3pt);
\draw[fill=black] (3,.3) circle (3pt);
\draw[fill=black] (3.5,0) circle (3pt);
\draw[fill=black] (3.5,.6) circle (3pt);
\end{tikzpicture}
\\[-5pt]
\begin{tikzpicture}
\node at (-.3,.3) {$x$};
\node at (.8,.6) {$y$};
\node at (.8,0) {$z$};
\node at (2.7,.3) {$x$};
\node at (3.8,.6) {$y$};
\node at (3.8,.0) {$z$};
\draw[thick] (0,.3) -- (.5,0); 
\draw[thick](0,.3)--(.5,.6);
\draw[thick] (3,.3) -- (3.5,0); 
\draw[thick](3,.3)--(3.5,.6);
\draw[thick,->] (1.2,.3)--(2.2,.3);
\draw[fill=white] (0,.3) circle (3pt);
\draw[fill=black] (.5,0) circle (3pt);
\draw[fill=black] (.5,.6) circle (3pt);
\draw[fill=white] (3,.3) circle (3pt);
\draw[fill=black] (3.5,0) circle (3pt);
\draw[fill=black] (3.5,.6) circle (3pt);
\end{tikzpicture}
\begin{tikzpicture}
\node at (-.3,.3) {$x$};
\node at (.8,.6) {$y$};
\node at (.8,0) {$z$};
\node at (2.7,.3) {$x$};
\node at (3.8,.6) {$y$};
\node at (3.8,.0) {$z$};
\draw[thick] (0,.3) -- (.5,0); 
\draw[thick](0,.3)--(.5,.6);
\draw[thick] (3,.3) -- (3.5,0); 
\draw[thick](3,.3)--(3.5,.6);
\draw[thick,->] (1.2,.3)--(2.2,.3);
\draw[fill=white] (0,.3) circle (3pt);
\draw[fill=black] (.5,0) circle (3pt);
\draw[fill=white] (.5,.6) circle (3pt);
\draw[fill=white] (3,.3) circle (3pt);
\draw[fill=black] (3.5,0) circle (3pt);
\draw[fill=white] (3.5,.6) circle (3pt);
\end{tikzpicture}
\\[-5pt]
\begin{tikzpicture}
\node at (-.3,.3) {$x$};
\node at (.8,.6) {$y$};
\node at (.8,0) {$z$};
\node at (2.7,.3) {$x$};
\node at (3.8,.6) {$y$};
\node at (3.8,.0) {$z$};
\draw[thick] (0,.3) -- (.5,0); 
\draw[thick](0,.3)--(.5,.6);
\draw[thick] (3,.3) -- (3.5,0); 
\draw[thick](3,.3)--(3.5,.6);
\draw[thick,->] (1.2,.3)--(2.2,.3);
\draw[fill=white] (0,.3) circle (3pt);
\draw[fill=white] (.5,0) circle (3pt);
\draw[fill=black] (.5,.6) circle (3pt);
\draw[fill=white] (3,.3) circle (3pt);
\draw[fill=white] (3.5,0) circle (3pt);
\draw[fill=black] (3.5,.6) circle (3pt);
\end{tikzpicture}
\begin{tikzpicture}
\node at (-.3,.3) {$x$};
\node at (.8,.6) {$y$};
\node at (.8,0) {$z$};
\node at (2.7,.3) {$x$};
\node at (3.8,.6) {$y$};
\node at (3.8,.0) {$z$};
\draw[thick] (0,.3) -- (.5,0); 
\draw[thick](0,.3)--(.5,.6);
\draw[thick] (3,.3) -- (3.5,0); 
\draw[thick](3,.3)--(3.5,.6);
\draw[thick,->] (1.2,.3)--(2.2,.3);
\draw[fill=white] (0,.3) circle (3pt);
\draw[fill=white] (.5,0) circle (3pt);
\draw[fill=white] (.5,.6) circle (3pt);
\draw[fill=white] (3,.3) circle (3pt);
\draw[fill=white] (3.5,0) circle (3pt);
\draw[fill=white] (3.5,.6) circle (3pt);
\end{tikzpicture}
\end{split}
\end{align}\end{linenomath*}
On the left-hand side of the first row of (\ref{diag_dual2}), all of the three sites $x,y,z$ are occupied by $1$-individuals, but after the update, the child of $x$ invading site $y$ and the resident $1$-individual at $y$ annihilate each other so that an empty site is created. The same holds at  site $z$. 
On the right-hand side of the first row of (\ref{diag_dual2}), site $y$ becomes occupied by a $1$-individual after the update, but pairwise annihilation still occurs at site $z$.

Now  the transposed linear transformation in  (\ref{voter-0}) defines a random walk with annihilation: 
\begin{linenomath*}\begin{align}\label{voter-2}
\xi\to \xi+\{x,y\}\times \{x\}\xi&\equiv  [\xi-\xi(x)\1_x-\xi(y)\1_y]+[\xi(x)+\xi(x)]\1_x+[\xi(x)+\xi(y)]\1_y\notag\\
&= [\xi-\xi(x)\1_x-\xi(y)\1_y]+[\xi(x)+\xi(y)]\1_y,
\end{align}\end{linenomath*}
so that the the individual at site $x$ moves to site $y$ subject to pairwise annihilation:
\begin{linenomath*}\begin{align}\label{diag_dual1}
\begin{split}
\begin{tikzpicture}
\node at (-.3,.3) {$x$};
\node at (.8,0.3) {$y$};
\node at (3.8,.3) {$y$};
\node at (2.7,.3) {$x$};
\draw[thick](0,.3)--(.5,.3);
\draw[thick](3,.3)--(3.5,.3);
\draw[thick,->] (1.2,.3)--(2.2,.3);
\draw[fill=black] (0,.3) circle (3pt);
\draw[fill=black] (.5,.3) circle (3pt);
\draw[fill=white] (3,.3) circle (3pt);
\draw[fill=white] (3.5,.3) circle (3pt);
\end{tikzpicture}
\begin{tikzpicture}
\node at (-.3,.3) {$x$};
\node at (.8,0.3) {$y$};
\node at (3.8,.3) {$y$};
\node at (2.7,.3) {$x$};
\draw[thick](0,.3)--(.5,.3);
\draw[thick](3,.3)--(3.5,.3);
\draw[thick,->] (1.2,.3)--(2.2,.3);
\draw[fill=white] (0,.3) circle (3pt);
\draw[fill=black] (.5,.3) circle (3pt);
\draw[fill=white] (3,.3) circle (3pt);
\draw[fill=black] (3.5,.3) circle (3pt);
\end{tikzpicture}\\[-5pt]
\begin{tikzpicture}
\node at (-.3,.3) {$x$};
\node at (.8,0.3) {$y$};
\node at (3.8,.3) {$y$};
\node at (2.7,.3) {$x$};
\draw[thick](0,.3)--(.5,.3);
\draw[thick](3,.3)--(3.5,.3);
\draw[thick,->] (1.2,.3)--(2.2,.3);
\draw[fill=black] (0,.3) circle (3pt);
\draw[fill=white] (.5,.3) circle (3pt);
\draw[fill=white] (3,.3) circle (3pt);
\draw[fill=black] (3.5,.3) circle (3pt);
\end{tikzpicture}
\begin{tikzpicture}
\node at (-.3,.3) {$x$};
\node at (.8,0.3) {$y$};
\node at (3.8,.3) {$y$};
\node at (2.7,.3) {$x$};
\draw[thick](0,.3)--(.5,.3);
\draw[thick](3,.3)--(3.5,.3);
\draw[thick,->] (1.2,.3)--(2.2,.3);
\draw[fill=white] (0,.3) circle (3pt);
\draw[fill=white] (.5,.3) circle (3pt);
\draw[fill=white] (3,.3) circle (3pt);
\draw[fill=white] (3.5,.3) circle (3pt);
\end{tikzpicture}
\end{split}
\end{align}\end{linenomath*}
On the left-hand side of the first row of (\ref{diag_dual1}), the $1$-individual moves to site $y$, but this individual and the resident $1$-individual at $y$ annihilate each other. We get two empty sites in the end. 
On the right-hand side of the first row of (\ref{diag_dual1}), site $x$ is empty and so has no effect on the type occupying $y$ from the update event.

In contrast to the Neuhauser--Pacala model where a death event is immediately followed by a birth event at each updating step, the dual process is defined by reversing the orders of birth events and  death events, and thus, is an invasion process. Also, it is possible to have multiple sites where a change of types occurs in both mechanisms defining the dual process.

To fully define the dual process in (\ref{dual+1}), we also need to specify the rates of the associated Poisson processes. 
The Poisson process corresponding to (\ref{ann-2}) jumps with the same rate as that for (\ref{ann}):
\begin{linenomath*}\begin{align}\label{ann-3}
\widehat{r}(\{y,z\}\times \{x\})\stackrel{\rm def}{=}r(\{x\}\times \{y,z\}),
\end{align}\end{linenomath*}
whereas the Poisson process corresponding to (\ref{voter-2}) jumps with rate 
\begin{linenomath*}\begin{align}\label{voter-3}
\widehat{r}(\{x,y\}\times \{x\})\stackrel{\rm def}{=}r(\{x\}\times \{x,y\}) .
\end{align}\end{linenomath*}

Taking expectations of both sides of (\ref{dual}), we have proved the following theorem when $E$ is a finite set. It still holds on infinite sets. 

\begin{thm}[\bf Parity duality]\label{thm:PDE}
For any irreducible kernel $(E,q)$ with a zero trace, finite subsets $A,B$ of $E$, and $t\geq 0$, we have 
\begin{align}\label{parityeq}
\P\big(\langle \1_B,\eta^A_t\rangle\equiv 1\big)=\P\big(\langle \widehat{\eta}^B_t,\1_A\rangle\equiv 1\big),
\end{align}
where $\equiv$ means an equality mod $2$. 
\end{thm}

\begin{rmk}
We can also view a population configuration $\eta$ as the set of sites occupied by $1$'s in $\eta$. In this interpretation, $\langle \1_B,\eta\rangle$ is the  same as $|B\cap \eta|$. 
Hence, (\ref{parityeq}) is also an equality for the probability that there are odd many sites with $1$'s in $B$ under $\eta^A_t$ and the probability that there are odd many sites with $1$'s in $A$ under $\widehat{\eta}^B_t$. 
This interpretation is common in the literature due to its usefulness for geometric considerations of population configurations.\mbox{}\quad\mbox{}
\hfill $\blacksquare$
\end{rmk}

The validity of Theorem~\ref{thm:PDE} on infinite sets 
can be obtained by various methods. For example, one can check the following Feynman--Kac duality by generator calculations. 
The approach is to consider the forward equation satisfied by 
\[
\phi_A(B)=\phi_B(A)\stackrel{\rm def}{=}\1_{\{\langle \1_B,\1_A\rangle\equiv 1\}}
\]
under the symmetric Neuhauser--Pacala model (with semigroup $P_t=\exp\{t\mathsf L^{\rm NP}\}$) and then to show that this forward equation is also a backward equation under the dual process (with semigroup $(Q_t=\exp\{t\mathsf L^{\rm dual}\}$):
\begin{linenomath*}\begin{align}\label{FK}\frac{d}{dt}P_t\phi_B(A) =P_t\mathsf L^{\mbox{\scriptsize NP}}\phi_B(A)
=Q_t\mathsf L^{\rm dual}\phi_A(B)
\end{align}\end{linenomath*}
so that 
\[
P_t\phi_B(A)=Q_t\phi_A(B).
\]
Here, $\mathsf L^{\rm NP}$ denotes the generator of the Neuhauser--Pacala model and $\mathsf L^{\rm dual}$ denotes the generator of its dual.
In more detail, to do the algebra for the proof of the second equality in (\ref{FK}), one needs the possible transitions and rates defining $\mathsf L^{\rm NP}$ and  the possible transitions and rates in (\ref{ann-2}), (\ref{voter-2}), (\ref{ann-3}) and (\ref{voter-3}) defining 
 $\mathsf L^{\rm dual}$. We omit the details here, but in Section~\ref{sec:dual}, we will illustrate this algebraic method by  the interacting diffusions defined at the end of Section~\ref{sec:LV}. 
See \cite[Proposition~1]{NP} for this calculation and \cite[Section~III.4]{Liggett} or \cite{JK14,SSV} for more general discussions. 

Harris's graphical representation \cite{Harris} gives a different  proof of Theorem~\ref{thm:PDE} on infinite sets. 
To motivate the representation, first we consider the case that $E$ is finite. We redefine $(\eta^A_t)$ in  (\ref{construct:linear}) as follows. For each matrix $J$ such that $\Id +J$ is a linear transformation in (\ref{ann}) or (\ref{voter}), we define a Poisson process $N(J)=\{N_t(J);t\geq 0\}$ with rate $r(J)$. We assume that these Poisson processes are independent. 
 Order the arrival times of $\sum_J N(J)$ as \begin{align}\label{def:Tn}
 0=T_0<T_1<T_2<T_3<\cdots,
 \end{align} 
and write $J_n$ for the matrix $J$ such that $T_n$ is an arrival time of $N(J)$. Then by thinning and superposition of Poisson processes, the process $(\eta^A_t)$ constructed before has the same distribution as the following process which we also denote by $(\eta^A_t)$: 
\begin{align}\label{etaA:v2}
\eta^A_t=(\Id+J_n)\cdots (\Id +J_2)(\Id+J_1)\1_A\quad\mbox{if }T_n\leq t<T_{n+1}. 
\end{align}
See also \cite[Chapter~2]{Norris} for the usual construction of continuous-time Markov chains. 

We make some simple observations for matrix products to interpret the definition in (\ref{etaA:v2}).  In general,
for matrices $M_1,\ldots,M_n$ with entries in $\{0,1\}$, write
\begin{align}\begin{aligned}
&\1_B^\top  M_nM_{n-1}\cdots M_1\1_A=\sum_{x_n,x_{n-1},\ldots,x_0\in E}\1_B(x_n)M_n(x_n,x_{n-1})
\cdots M_1(x_1,x_0)\1_A(x_0).\label{MBA}
\end{aligned}\end{align}
Since the matrices have entries only in $\{0,1\}$, the summands are $\{0,1\}$-valued with
\begin{align}\label{path}
\1_B(x_n)M_n(x_n,x_{n-1})M_{n-1}(x_{n-1},x_{n-2})\cdots M_1(x_1,x_0)\1_A(x_0)=1
\end{align}
if and only if 
\[
1=\1_A(x_0)=M_1(x_1,x_0)=\cdots=M_n(x_n,x_{n-1})=\1_B(x_n).
\]
If an $E\times E$ matrix  $M$ with entries in $\{0,1\}$ is identified with a collection of oriented edges $x\to y$ in $E$ such that $M(y,x)=1$ and vice versa, then the product in (\ref{path}) corresponds to the following oriented path:
\begin{align}\label{path1}
x_0\to x_1 \to \cdots \to x_n.
\end{align}
We stress that $M(y,x)=1$ corresponds to the oriented path $x\to y$ rather than to the reversed one, since matrices are now multiplied from the right to the left as in (\ref{etaA:v2}). 
In summary, \mbox{$\1_B^\top  M_n\cdots M_1\1_A$} is the number of oriented paths from $A$ to $B$.

For the purpose of (\ref{etaA:v2}), we set all of the matrices $M_m$ in (\ref{MBA}) to be of the following form  
\[
M\equiv \Id+J ,
\]
where the sum is read modulo 2 entry-wise as usual. Then on $\{T_n\leq t<T_{n+1}\}$,  a nonzero term 
\[
\1_B(x_n)(\Id +J_n)(x_n,x_{n-1})(\Id +J_{n-1})(x_{n-1},x_{n-2})\cdots (\Id +J_1)(x_1,x_0)\1_A(x_0)
\]
in the representation of $\eta^A_t$ in (\ref{etaA:v2})
can be identified with an oriented path from $(x_0,0)$ to $(x_n,t)$ in $E\times \R_+$. Over $[0,t)$, this space-time path is defined by the pair $(x_j,s)$ when $s\in [T_j,T_{j+1})$. Hence, (\ref{etaA:v2}) can be written as
\begin{align}\label{etaA:v3}
\eta_t^A(x)\equiv \#\big\{\mbox{oriented space-time paths from $A\times \{0\}$ to $\{x\}\times \{t\}$}\big\}
\end{align}
for all $x\in E$.
We have obtained the graphical representation of $(\eta^A_t)$.
Moreover, for every fixed $t\in (0,\infty)$, the reversely oriented space-time paths between $A\times \{0\}$ and $B\times \{t\}$ define $\widehat{\eta}^{t,B}_t$ by the same counting procedure as in (\ref{etaA:v3}). 
In particular, 
 $\langle \widehat{\eta}^{t,B}_t,\1_A\rangle=\langle \1_B,\eta^A_t\rangle$.
 
\begin{example}
Let us visualize the graphical representation of $(\eta^A_t)$ in the following simple setting. 
 For convenience, we write 
\begin{tikzpicture}
\draw[fill=red,draw=none] (0,0) circle (3pt);
\end{tikzpicture} $=1$ and 
\begin{tikzpicture}
\draw[fill=blue,draw=none] (0,0) circle (3pt);
\end{tikzpicture} $=0$. If $E=\{x,y,z\}$, the initial condition $\1_y$ after the update $\{x\}\times \{y,z\}$ can be visualized by
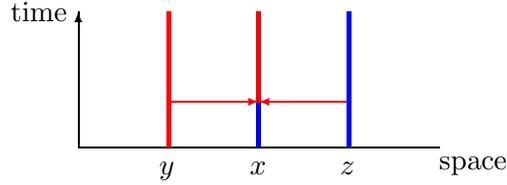
\begin{center}
\setlength{\unitlength}{.6cm}
\vspace{-3.5cm}
\hspace{-3cm}
\begin{picture}(23,10)(0,-6)
\put(10,-5){\line(1,0){8}}
\put(10,-5){\vector(0,1){3}}
\put(18,-5.5){space}
\put(13.8,-5.6){$x$}
\put(11.8,-5.6){$y$}
\put(15.8,-5.6){$z$}
\put(8.5,-2.2){time}
\color{red}
\put(12,-4){\vector(1,0){2}}
\put(16,-4){\vector(-1,0){2}}
\color{blue}
\linethickness{.5mm}
\put(14,-5){\line(0,1){1}}
\put(16,-5){\line(0,1){3}}
\color{red}
\linethickness{.5mm}
\put(12,-5){\line(0,1){3}}
\put(14,-4){\line(0,1){2}}
\end{picture}
\end{center}
Similarly, for $E=\{x,y\}$,
the initial condition $\eta_0=\1_y$
after the voting update of $\eta\mapsto \eta+\{x\}\times \{x,y\}\eta$ can be visualized by
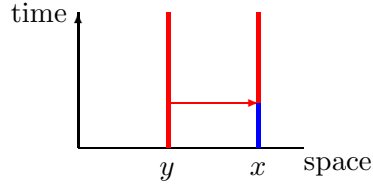
\begin{center}
\setlength{\unitlength}{.6cm}
\vspace{-3.5cm}
\hspace{-3cm}
\begin{picture}(23,10)(0,-6)
\put(10,-5){\line(1,0){5}}
\put(10,-5){\vector(0,1){3}}
\put(15,-5.5){space}
\put(13.8,-5.6){$x$}
\put(11.8,-5.6){$y$}
\put(8.5,-2.2){time}
\color{red}
\put(12,-4){\vector(1,0){2}}
\color{blue}
\linethickness{.5mm}
\put(14,-5){\line(0,1){1}}
\color{red}
\linethickness{.5mm}
\put(12,-5){\line(0,1){3}}
\put(14,-4){\line(0,1){2}}
\end{picture}
\end{center}
\hfill $\square$
\end{example}

Now we consider the case that $E$ is an infinite set.
We define independent Poisson processes $N(J)$ as before. 
 In this case, since $\sum_J r(J)$ is infinite,  the arrival times of $\sum_J N(J)$ cannot be ordered as in (\ref{def:Tn}). Moreover,  (\ref{etaA:v2}) shows that we have to deal with an infinite product of matrices now. 
These issues can be easily circumvented since the following picture holds with probability one: \emph{Locally}, there are only finitely many such arrival times which trigger nontrivial updates, and the infinite product of matrices can be reduced to a finite product of matrices.

Let us define these reductions precisely. For a general $E\times E$ matrix $M$ with entries in $\{0,1\}$, we define the {\bf range} of $M$ by 
\begin{align}\begin{aligned}
\mathcal R(M)&\stackrel{\rm def}{=}\{y:M(y,x)=1\mbox{ for some }x\in E\mbox{ and $x\neq y$}\}\\
&\qquad \cup \{x:M(y,x)=1\mbox{ for some }y\in E \mbox{ and $y\neq x$}\}.
\end{aligned}\end{align}
In terms of the above interpretation, $\mathcal R(M)$ is the set of endpoints of oriented edges in $E$ defined by the nonzero entries of $M$, but we disregard oriented edges which are self-loops. 
Write $0=T_0(J)<T_1(J)<\cdots $ for the arrival times of $N(J)$. Then the local picture follows from the observation that for all $x\in E$,
\begin{align}\label{Poisson-finite}
\sum_{J:x\in \mathcal R(\Id +J)}r(J)<\infty\Longrightarrow \sum_{J:x\in \mathcal R(\Id +J)}N_t(J)<\infty\mbox{ a.s. for all $t\geq 0$.}
\end{align}
We draw oriented edges  induced by $\Id+J$ at all of the times $T_n(J)$ and require that these oriented edges are not self-loops.
Then the property of the Poisson processes in (\ref{Poisson-finite}) implies that up to a fixed time $t\in (0,\infty)$, there are only finitely many oriented edges in $E$ which are not self-loops and use $x$ as the endpoints. Therefore, a definition of $\eta^A$ by (\ref{etaA:v3}) and an analogous definition for $\widehat{\eta}^{t,B}$
 apply again. We still have $\langle \widehat{\eta}^{t,B}_t,\1_A\rangle=\langle \1_B,\eta^A_t\rangle$.

The usefulness of the parity duality in Theorem~\ref{thm:PDE} follows from the fact that parity events $\{\eta;\langle \1_B,\eta\rangle\equiv 1\}$, $B\subseteq E$, uniquely determine a measure.

\begin{prop}\label{prop:measure}
For finite measures $\nu_1$ and $\nu_2$ with the same total mass,
$\nu_1\{\eta;\langle\1_B,\eta\rangle\equiv 1\}=\nu_2\{\eta;\langle  \1_B,\eta\rangle\equiv 1\}$ for all finite subsets $B$ of $E$ implies that $\nu_1=\nu_2$.
\end{prop}
\begin{proof}[\bf Proof]
First, for all finite subsets $A$ of $E$, the test functions $\prod_{x\in A}[2\eta(x)-1]$ can be written as
\begin{linenomath*}\begin{align*}
\prod_{x\in A}[2\eta(x)-1]=&(-1)^{\langle \1_A,\eta^\complement\rangle}=1-2\1_{\{\eta;\langle\1_A, \eta^\complement\rangle\equiv 1\}}.
\end{align*}\end{linenomath*}
($\prod_{x\in \varnothing}[2\eta(x)-1]$ and other products over an empty set used below are equal to $1$ by convention.)
Hence, the assumption on $\nu_1$ and $\nu_2$ shows that for all finite subsets $A$ of $E$,  
\[
\int \prod_{x\in A}[2\eta(x)-1]\nu_1(d\eta)=\int \prod_{x\in A}[2\eta(x)-1]\nu_2(d\eta). 
\]
Also, expanding $\prod_{x\in A}[2\eta(x)-1]$ yields
\begin{linenomath*}\begin{align*}
2^{|A|}\prod_{x\in A}\eta(x)=\prod_{x\in A}[2\eta(x)-1]-\sum_{B:B\subsetneq A}2^{|B|}(-1)^{|A-B|}\prod_{x\in B}\eta(x).
\end{align*}\end{linenomath*}
Hence, by the last two displays and an induction on $|A|$, we deduce that  
\[
\int \prod_{x\in A}\eta(x)\nu_1(d\eta)=\int \prod_{x\in A}\eta(x)\nu_2(d\eta)
\]
for all finite subsets $A$ of $E$. (We use the assumption that $\nu_1$ and $\nu_2$ have the same total mass for the initial induction step with $|A|=0$.)  Since $\eta(x)$ is $\{0,1\}$-valued for all $x\in E$, the foregoing equality can be restated as 
\[
\nu_1\left(\bigcap_{x\in A}\{\eta;\eta(x)=1\}\right)=\nu_2\left(\bigcap_{x\in A}\{\eta;\eta(x)=1\}\right) 
\]
and we have $\{\eta;\eta(x)=1\}=\{\eta;\eta(x)=0\}^\complement$ for all $x\in E$. Hence, applying  the inclusion-exclusion principle for the sets $\{\eta;\eta(x)=1\}$ shows that the two measures $\nu_1$ and $\nu_2$ are equal. 
\end{proof}

\subsection{Invariance of equiparity coexistence}\label{sec:NP_inv}
If we start the symmetric Neuhauser--Pacala model with the Bernoulli product measure $\beta_{1/2}$ with density $1/2$, then the corresponding parity duality from Theorem~\ref{thm:PDE} can be written as
\begin{linenomath*}\begin{align}\label{beta12}
\P_{\beta_{1/2}}(\langle \1_B,\eta_t\rangle\equiv 1)=\P\left(\sum_{m=1}^{\langle \widehat{\eta}^B_t,\1\rangle}X_m\equiv 1\right). 
\end{align}\end{linenomath*}
Here, $X_1,X_2,\ldots$ are i.i.d. $\{0,1\}$-valued Bernoulli variables with mean $1/2$ and independent of $\widehat{\eta}^B_t$. They come from the random spins at all sites under $\beta_{1/2}$.
By conditioning the event on the right-hand side of (\ref{beta12}) on $\widehat{\eta}^B_t$, one sees that
\begin{linenomath*}\begin{align}\label{beta121}
\P_{\beta_{1/2}}(\langle \1_B,\eta_t\rangle\equiv 1)=\frac{1}{2}\P(\widehat{\eta}^B_t\neq \mathbf 0).
\end{align}\end{linenomath*}
Hence, the problem of proving existence of the limit of $\P_{\beta_{1/2}}(\langle \1_B,\eta_t\rangle\equiv 1)$ as $t\to\infty$  is equivalent to the problem of calculating the survival probability in $(\widehat{\eta}^B_t)$.

In the rest of Section~\ref{sec:spin}, we give a brief discussion of calculating the limiting distribution of $(\eta^A_t)$ as $t\to\infty$, and the method is to refine what we just observed under the Bernoulli initial condition $\beta_{1/2}$. 
In this regard, the reader may wish to recall the discussion of (\ref{p0*}). 
Since $\lambda =1$ and $\alpha_{01}=\alpha_{10}=\alpha$ for the symmetric Neuhauser--Pacala model, it seems reasonable to expect that by suitable definitions,
equilibria of the model can be related to the value $p_0^*$ defined by (\ref{p0*}), which is $1/2$ under the present choice of $\lambda,\alpha_{01},\alpha_{10}$. This number $1/2$ is ubiquitous in the rest of this section.

The following theorem is our first theorem about $1/2$ in equilibria of the model.

\begin{thm}\label{thm:1}
If $\alpha=0$, then $\beta_{1/2}$ is invariant for the symmetric Neuhauser--Pacala model.
\end{thm}
\begin{proof}[\bf Proof]
Recall that the parameter $\alpha$ from (\ref{def:alpha}) enters the weights of the two mechanisms in $(\eta^A_t)$ (see (\ref{ann}) and (\ref{voter})) and the dual process $(\widehat{\eta}^B_t)$ (see (\ref{ann-3}) and (\ref{voter-3})). 
If $\alpha=0$, then there is no random walk with annihilation in $(\widehat{\eta}^B_t)$. In this case, observe that  $(\widehat{\eta}^B_t)$ cannot die out unless $B=\varnothing$. Indeed, any empty site has no effect on any other sites. Algebraically, this is attributable to the fact that $0$ is the additive identity in $\mathbb Z_2$. On the other hand, any $1$-site $x$ can only change its neighbors in three ways:
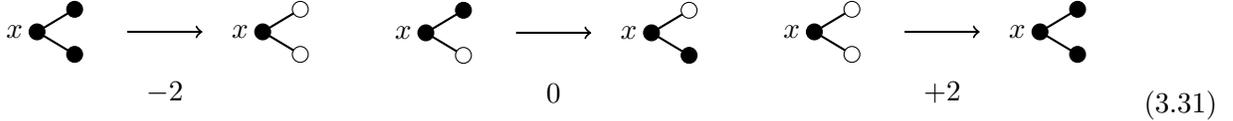
\begin{linenomath*}\begin{align}
\begin{tikzpicture}
\node at (-.3,.3) {$x$};
\node at (4.5,0) {};
\node at (2.5,.8) {};
\node at (2.7,.3) {$x$};
\node at (7.5,0) {};
\node at (1.7,-.5) {$-2$};
\draw[thick] (0,.3) -- (.5,0); 
\draw[thick](0,.3)--(.5,.6);
\draw[thick] (3,.3) -- (3.5,0); 
\draw[thick](3,.3)--(3.5,.6);
\draw[thick,->] (1.2,.3)--(2.2,.3);
\draw[fill=black] (0,.3) circle (3pt);
\draw[fill=black] (.5,0) circle (3pt);
\draw[fill=black] (.5,.6) circle (3pt);
\draw[fill=black] (3,.3) circle (3pt);
\draw[fill=white] (3.5,0) circle (3pt);
\draw[fill=white] (3.5,.6) circle (3pt);
\end{tikzpicture}
\hspace{-3cm}
\begin{tikzpicture}
\node at (-.3,.3) {$x$};
\node at (4.5,0) {};
\node at (2.5,.8) {};
\node at (2.7,.3) {$x$};
\node at (7.5,0) {};
\node at (1.7,-.5) {$0$};
\draw[thick] (0,.3) -- (.5,0); 
\draw[thick](0,.3)--(.5,.6);
\draw[thick] (3,.3) -- (3.5,0); 
\draw[thick](3,.3)--(3.5,.6);
\draw[thick,->] (1.2,.3)--(2.2,.3);
\draw[fill=black] (0,.3) circle (3pt);
\draw[fill=white] (.5,0) circle (3pt);
\draw[fill=black] (.5,.6) circle (3pt);
\draw[fill=black] (3,.3) circle (3pt);
\draw[fill=black] (3.5,0) circle (3pt);
\draw[fill=white] (3.5,.6) circle (3pt);
\end{tikzpicture}
\hspace{-3cm}
\begin{tikzpicture}
\node at (-.3,.3) {$x$};
\node at (4.5,0) {};
\node at (2.5,.8) {};
\node at (2.7,.3) {$x$};
\node at (7.5,0) {};
\node at (1.7,-.5) {$+2$};
\draw[thick] (0,.3) -- (.5,0); 
\draw[thick](0,.3)--(.5,.6);
\draw[thick] (3,.3) -- (3.5,0); 
\draw[thick](3,.3)--(3.5,.6);
\draw[thick,->] (1.2,.3)--(2.2,.3);
\draw[fill=black] (0,.3) circle (3pt);
\draw[fill=white] (.5,0) circle (3pt);
\draw[fill=white] (.5,.6) circle (3pt);
\draw[fill=black] (3,.3) circle (3pt);
\draw[fill=black] (3.5,0) circle (3pt);
\draw[fill=black] (3.5,.6) circle (3pt);
\end{tikzpicture}
\hspace{-4cm}
\end{align}\end{linenomath*}
Here, the integer below each figure  is the difference of the numbers of filled sites between the right-hand graph and the left-hand graph. 

Now, by (\ref{beta12}) and (\ref{beta121}), we see that $\P_{\beta_{1/2}}(\langle \1_B,\eta_t\rangle\equiv 1)=1/2$ for all finite nonempty subsets $B$ of $E$. 
On the other hand, a symmetric argument shows that $\beta_{1/2}\{\eta;\langle \1_B,\eta\rangle\equiv 1\}=1/2$  for all finite nonempty subsets $B$ of $E$. 
By Proposition~\ref{prop:measure}, we conclude that $\P_{\beta_{1/2}}(\eta_t\in \cdot)=\beta_{1/2}$ for all $t\geq 0$, as required.
\end{proof}

There are two natural questions stemming from this simple theorem.

\begin{que}
Is $\beta_{1/2}$ also the unique limiting distribution whenever we start with initial conditions different from the all-$1$ configuration and the all-$0$ configuration?  \hfill $\blacksquare$
\end{que}

\begin{que}
What can be said about the set of invariant distributions of the symmetric Neuhauser--Pacala model if $\alpha\in (0,1)$? \hfill $\blacksquare$
\end{que}

These questions are closely related to 
the following  conjecture
due to Neuhauser and Pacala~\cite[Conjecture~1]{NP}. 

\begin{conj}\label{conj:NP}
There is coexistence in the symmetric Neuhauser--Pacala model on $\mathbb Z^d$ for any $d\geq 2$. 
\end{conj}

If $\alpha>0$ and $q(x,y)$ defines the nearest-neighbor random walk on $\Z$, coexistence in the sense that one can see the populations of both types occupy a fixed region after large times fails with probability one, and hence, clustering holds. 
See \cite[Theorem~2 (b)]{NP} for the precise statement. On the other hand, in view of the solution (\ref{p0*}) under the classical Lotka--Volterra model and the convergence of the symmetric Neuhauser--Pacala model 
on large complete graphs (recall the discussion below (\ref{def:flip})), it seems reasonable to expect that a mean-field phenomenon should be valid under these stochastic spatial generalizations in general. Coexistence in suitable notions should hold except in very few, if not pathological, cases. Some resolutions of Conjecture~\ref{conj:NP} are achieved in \cite{CP:Rescaled,CMP,CP:Coexistence,SS}, since \cite[Theorem~1]{NP} takes the first few steps to do so. 
Yet to the knowledge of the authors, the full resolution of Conjecture~\ref{conj:NP} remains open.

The role of $1/2$ in equilibria can be illustrated more generally by the following theorem. 

\begin{thm}\label{thm:equiparity1}
Suppose that $(\eta_t)$ has initial condition given by $\beta_{u}$ for $u\in (0,1)$ and that along some subsequence $(t_k)$ tending to infinity,
\begin{linenomath*}\begin{align}\label{def:ZB}
&\langle \widehat{\eta}^{B}_{t_k},\1\rangle\xrightarrow[t_k\to\infty]{{\rm (d)}}Z_B\quad \mbox{in}\;  \mathbb Z_+\cup \{+\infty\}
\end{align}\end{linenomath*}
for some random variable $Z_B$. Then
\begin{linenomath*}\begin{align*}
\P_{\beta_u}\big(\langle \1_B, \eta_{t_k}\rangle\equiv 1\big)\xrightarrow[t_k\to\infty]{} \frac{1}{2}\E\left[1-(1-2u)^{Z_B}\right]
\end{align*}\end{linenomath*}
with the convention $0^0=1$ when $u=1/2$. 
\end{thm}

This theorem is a generalization of what we discuss above for the case $u=1/2$ in Theorem~\ref{thm:1}. 
The following corollary is immediate from Theorem~\ref{thm:equiparity1}.

\begin{cor}\label{cor:EVUG}
Suppose that (\ref{def:ZB}) holds and $u\neq 1/2$.
Then $\P_{\beta_u}\big(\langle \1_B, \eta_{t_k}\rangle\equiv 1\big)\to \tfrac{1}{2}$ if and only if 
the dual process survives:
\[
\lim_{t_k\to\infty}\P(\langle \widehat{\eta}^B_{t_k},\1\rangle\geq 1)=1
\]
and there is {\bf extinction versus unbounded growth}  in the dual process:
\begin{align}\label{EVUG}
\lim_{t_k\to\infty}\P(1\leq \langle \widehat{\eta}^B_{t_k},\1\rangle \leq L)=0,\quad\forall\; L\in \mathbb N.
\end{align}
\end{cor}

The notion of extinction versus unbounded growth is introduced by Sturm and Swart in~\cite{SS}. We will discuss this notion in more detail later on.

The proof of Theorem~\ref{thm:equiparity1} is an application of the parity duality and the following lemma. 

\begin{lem}[Parity deviation]\label{lem:equiparity}
Let $X_1,X_2,\cdots, X_N$ be independent $\mathbb Z_+$-valued random variables with
$\P(X_m\equiv 1)=u_m$.
Then it holds that
\begin{linenomath*}\begin{align*}
\P\left(\sum_{m=1}^NX_m\equiv 0\right)-\P\left(\sum_{m=1}^NX_m\equiv 1\right)=\prod_{m=1}^N(1-2u_m).
\end{align*}\end{linenomath*}
\end{lem}
\begin{proof}[\bf Proof]
The identity follows upon writing the left-hand side as $\E\big[(-1)^{\sum_{m=1}^NX_m}\big]$.
\end{proof}

\begin{proof}[\bf Proof of Theorem~\ref{thm:equiparity1}]
With a slight abuse of notation, we write also  $\beta_u$ as a random configuration independent of $(\widehat{\eta}_t^B)$. 
By the parity duality in Theorem~\ref{thm:PDE},
we have
\begin{linenomath*}\begin{align*}
\lim_{k\to\infty}\P_{\beta_u}\big(\langle \1_B, \eta_{t_k}\rangle\equiv 1\big)=\lim_{k\to\infty}\P\big(\langle  \widehat{\eta}^B_{t_k}, \beta_u\rangle\equiv 1\big)&=\lim_{k\to\infty}\frac{1}{2}\E\left[1-(1-2u)^{\langle \widehat{\eta}^B_{t_k},\1\rangle}\right]\\
&=\frac{1}{2}\E\left[1-(1-2u)^{Z_B}\right]
\end{align*}\end{linenomath*}
by (\ref{def:ZB}), 
where the second equality follows from Lemma~\ref{lem:equiparity}  by conditioning on $\widehat{\eta}^B_{t_k}$ since 
\[
\P\left(\sum_{m=1}^NX_m\equiv 1\right)=\frac{1}{2}\left[1-\P\left(\sum_{m=1}^NX_m\equiv 0\right)+\P\left(\sum_{m=1}^NX_m\equiv 1\right)\right]=\frac{1}{2}\left(1-\prod_{m=1}^N(1-2u_m)\right).
\] 
\end{proof}

For Theorem~\ref{thm:equiparity1}, further generalizations of the initial conditions of $(\eta_t)$ are possible. 
For example, on a finite set, $(\beta_u)_{0<u<1}$ can explicitly generate uniform distributions over sets of configurations with fixed numbers of $1$'s \cite[pages~655--656]{Chen} and links to distributions defined by exchangeable random variables over sites by de Finetti's representation theorem \cite[Theorem~11.10]{Kal}.
More generally, the proof of Theorem~\ref{thm:equiparity1} can be modified in an obvious way to show a similar formula if we use Bernoulli product measures with site-dependent densities, where the densities $\neq 0,1$ are bounded away from $0$ and $1$.

Theorem~\ref{thm:equiparity1} and Lemma~\ref{lem:equiparity} should provide the basic idea behind the related methods in \cite{ABP,NP,SS}. 
There the method to resolve the main technical issue for not starting with Bernoulli measures is to identify `Bernoulli subsystems' which have sizes growing to infinity as $t\to\infty$. 
The frameworks along this direction in \cite{ABP,NP,SS} can be roughly fit into the following principle,
where we generalize the setting of Lemma~\ref{lem:equiparity} to conditional probabilities in a time-dependent setting. Let $(\F_t)$ be a filtration and $\mathcal I_t,\mathcal B_t$ be $\mathcal F_t$-measurable subsets of $\mathbb N$. 
For each $t\geq 0$, the random variables $X_m$ in Lemma~\ref{lem:equiparity} are now replaced by a collection $\{X_m(t)$; $m\in \mathcal I_t\}$ of $\{0,1\}$-valued random variables. 
We assume that $X_m(t)$, $m\in  \mathcal B_t$, are conditionally independent given $\F_t$ and $\sum_{m\in \mathcal I_t\setminus \mathcal B_t}X_m(t)\in \F_t$. Then Lemma~\ref{lem:equiparity} implies
\begin{linenomath*}\begin{align}
\label{equiparity}
&\quad \;\P\left(\left.\sum_{m\in \mathcal I_t}X_m(t)\equiv 1\right|\F_t\right)\\
& =\P\left(\left.\sum_{m\in \mathcal B_t}X_m(t)\equiv 1-\sum_{m\in \mathcal I_t\setminus\mathcal B_t}X_m(t)\right|\F_t\right)\notag\\
&=\left\{
\begin{array}{ll}
\displaystyle \frac{1}{2}\left(1-\prod_{m\in \mathcal B_t}\big(1-u_m(t)\big)\right),\quad \mbox{ on }\left\{\sum_{m\in \mathcal I_t\setminus\mathcal B_t}X_m(t)\equiv 0\right\},\\
\vspace{-.2cm}\\
\displaystyle \frac{1}{2}\left(1+\prod_{m\in \mathcal B_t}\big(1-u_m(t)\big)\right),\quad \mbox{ on }\left\{\sum_{m\in \mathcal I_t\setminus\mathcal B_t}X_m(t)\equiv 1\right\},
\end{array}
\right.\label{product}
\end{align}\end{linenomath*}
where
\[
u_m(t)=\P(X_m(t)\equiv 1|\F_t).
\]
Hence, to make the conditional probability in (\ref{equiparity}) get close to $1/2$ with high  conditional probability, it is necessary to verify that  with high probability, $|\mathcal B_t|$ is large and $u_m(t)$ are bounded away from $0$ and $1$.

We close the discussion of this section with a very brief account of two applications of (\ref{equiparity})--(\ref{product}) in \cite{ABP,NP,SS}. There, the underlying spaces are integer lattices.

\paragraph{\small \bf Extinction versus unbounded growth.}
The method in \cite[Section~3]{SS} takes the following route. We consider the parity duality in the following form: For fixed $s>0$,
\begin{align}\label{EVUG-2}
\P(\langle \1_B,\eta^A_{s+t}\rangle\equiv 1)=\P(\langle \widehat{\eta}^B_t,\eta^A_s\rangle\equiv 1)=\E\big[\P(\langle \1_C,\eta^A_s\rangle\equiv 1)\big|_{\1_C=\widehat{\eta}^B_t}\big],
\end{align}
where $(\eta^A_s)$ and $(\widehat{\eta}^B_t)$ are independent as above.  
The first step is to show that if there are sufficiently many types of matrices $J$ such that the corresponding Poisson processes with rates $r(J)$ can trigger a change in 
$\langle \1_C,\eta^A_r\rangle$ for $r\in [0,s]$ by \emph{one} occurrence, then  we have
\[
\P(\langle \1_C,\eta^A_s\rangle\equiv 1)\simeq \frac{1}{2}.
\]
These matrices $J$'s are chosen such that $\langle \1_C,(\Id+J)\1_A\rangle=1$, and their ranges are sufficiently far apart so that one can write $\langle \1_C,\eta^A_s\rangle\equiv 1$ as a sum of these inner products and other terms.
Then  the principle in (\ref{equiparity}) comes into play because the Bernoulli variables $X_m(t)$, $m\in\mathcal B_t$,  are defined from conditioning these Poisson processes with rates $r(J)$ to ring at most once by time $s$.

The second step then aims to find the good $A$'s by using appropriate initial conditions and the good $C$'s for (\ref{EVUG-2})
by using the extinction versus unbounded growth condition in Corollary~\ref{cor:EVUG}. (The following initial conditions $\mu$ are enough for the required convergence: $\mu$ is translation invariant and is supported in the set of configurations $\1_A$ where $\eta^A_t$ can be any configuration with positive probability when restricted to an arbitrary finite set for every $t>0$.)
See \cite[Section~3.2]{SS} for these two steps.

The work in \cite{SS} also shows conditions for the extinction versus unbounded growth condition in terms of the survival of $(\eta_t)$ and
a certain recurrence property of $(\widehat{\eta}_t)$ when $\alpha\in (0,1)$ \cite[Sections~3.3 and 3.5]{SS}.  \hfill $\blacksquare$

\paragraph{\small \bf Domination by oriented percolation.}
The method in \cite{NP}, extended from \cite{ABP}, uses the following parity duality:
\begin{align}\label{DOP}
\P(\langle \1_B,\eta^A_{2t}\rangle\equiv 1)=\P(\langle \widehat{\eta}^B_t,\eta^A_t\rangle\equiv 1),
\end{align}
where $(\widehat{\eta}^B_t)$ is independent of $(\eta^A_t)$. ((\ref{DOP}) is a simple consequence of the Markov property of $(\eta^A_t)$ and the parity duality in Theorem~\ref{thm:PDE}.)
We identify configurations with sets of vertices occupied by $1$-individuals.
Then one writes $\langle \widehat{\eta}^B_t,\eta^A_t\rangle=|\widehat{\eta}^B_t\cap \eta^A_t|$ as a sum of the $\{0,1\}$-valued random variables
\[
X_m(t)\equiv |\widehat{\eta}^B_t\cap \eta^A_t\cap E_m|,\quad m\in \mathcal I_t=\mathbb Z,
\]
where $\{E_m\}$ is a partition of the whole space. In this case, the method in \cite{ABP} uses oriented percolations \cite{Durrett, Durrett_note} to show that the products in (\ref{product}) are close to $1/2$ for suitably chosen $\mathcal B_t$. \hfill $\blacksquare$

\section{Interacting diffusions on lattices}\label{sec:interacting}
In this section, we turn to coexistence in the system of interacting Wright-Fisher diffusions 
defined in (\ref{p:SDE}), where from now on we set $N=1$.

In addition to the connection to the Lotka--Volterra model explained in Section~\ref{sec:LV}, the model can be considered by itself as a model for the evolution of gene frequencies in a spatially structured two-type population subject to selection, where $p_t(x)$ describes the proportion of $0$-individuals at site $x\in\Z^d$ and time $t\ge0$. 
In this interpretation, at each site $x$, for $\mu<1$ there is (directional) selection in favor of the $0$-type if $s >0$ and in favor of the $1$-type if $s < 0$. 
 If $\mu>1$, then a type which is locally rare (corresponding to $p_t(x)$ being close to $0$ or $1$) has a selective advantage respectively disadvantage according to whether $s >0$ or $s < 0$, 
 thus for $s >0$ we have selection in favor of heterozygosity and for $s<0$ in favor of homozygosity. (Here, the term `heterozygosity' refers to the degree of genetic variation in the population.) 
It is easy to see that $(1-p_t)$ satisfies \eqref{p:SDE} with $s$ and $\mu$ replaced by $(-s)(1-\mu)$ and $\frac{\mu}{\mu-1}$, respectively.
In particular, if $\mu=2$, then $(1-p_t)$ follows the same dynamics as $(p_t)$, thus the two types evolve symmetrically.
Also note that for the neutral case $s=0$, the system \eqref{p:SDE} reduces to the well-known stepping stone model, see \cite{Shiga}.

It is natural to conjecture that long-term coexistence of the two types is possible in the heterozygosity (`balancing') selection case, at least if $s$ is large enough.
Indeed, based on comparison arguments with oriented percolation, in \cite[Theorem~1.4]{BEM} a coexistence result in the following sense is proven: 
for $\mu>1$ and fixed small $\vep>0$, there exists $s_0\ge0$ such that for all $s>s_0$ and all initial conditions $p_0$ with $p_0(x)\in(\varepsilon, 1-\varepsilon)$ for all $x\in\Z^d$, we have
\[
\liminf_{t\to\infty}\mathbb{P}\left(\varepsilon<p_t(x)<1-\varepsilon\right)>0,\quad\forall\;x\in \Z^d.
\]
Note that \cite[Theorem~1.4]{BEM} 
does not claim that coexistence fails for $s<s_0$, nor does it specify the value of $s_0$.
But in fact, one would conjecture the following (see \cite[Conjectures 2.2 and 2.3]{BEM}):
\begin{conj}\label{conj:BEM}
Suppose $\mu>1$. There exists a critical value $s_0\ge0$ such that we have coexistence for $s>s_0$ and non-coexistence for $s<s_0$. 
\end{conj} 

As in Conjecture~\ref{conj:NP}, the full resolution of this conjecture is still open to the knowledge of the authors.

Our discussion below will continue the point of view from Section~\ref{sec:spin} and relate coexistence to the question of survival for a suitable dual process.  
This time, the duality we can use 
is a {{\bf moment duality} extended from the well-known stepping stone model \cite{Shiga}. 
It requires a bit of algebra to see and, unfortunately, seems to be available for restricted values of the parameters $s$ and $\mu$ only.
In the `symmetric' case $\mu=2$, the moment dual is a {\bf branching annihilating random walk}, which is a non-monotone system and difficult to analyze mathematically. 
Based on non-rigorous results in the physics literature, \cite{BEM} actually conjectures the value of $s_0$ in Conjecture \ref{conj:BEM} to be equal to $0$ in dimensions $d\ge2$ and to be strictly positive in $d=1$.

\subsection{Duality and coexistence}\label{sec:dc}
We start with the particular case $s=0$, the stepping stone model, for which there is a well-known moment duality due to Shiga \cite{Shiga_86} with a system of coalescing random walks $(\xi_t)=(\xi_t(x); x\in\Z^d)$ on $\mathbb Z_+^{\Z^d}$ defined with the following rates:
\begin{linenomath*}\begin{align}\label{def:crw}
\begin{split}
&\mbox{migration:}\left\{
\begin{array}{ll}
\xi(x)\to \xi(x)-1\\
\vspace{-.4cm}\\
\xi(y)\to \xi(y)+1
\end{array}
\right.\mbox{ with rate }\xi(x)m_{xy},\\
&\hspace{.1cm}\mbox{coalescence: } \xi(x)\to \xi(x)-1\quad \mbox{ with rate }{\frac{\xi(x)[\xi(x)-1]}{2}}.
\end{split}
\end{align}\end{linenomath*}
The duality reads as follows: for each $p_0\in[0,1]^{\Z^d}$ and initial condition $\xi_0\in\mathbb Z_+^{\Z^d}$ such that $\sum_{x\in\Z^d}\xi_0(x)<\infty$, we have
\begin{linenomath*}\begin{align}\label{eq:moment_duality_coal}
\E\left[\prod_{x\in\Z^d}p_t(x)^{\xi_0(x)}\right]=\E\left[\prod_{x\in\Z^d}p_0(x)^{\xi_t(x)}\right].
\end{align}\end{linenomath*}
The migration mechanism in (\ref{def:crw}) shows that individuals at site $x$ can migrate to site $y$ with rate $m_{xy}$,
whereas the coalescence mechanism shows that only two of the individuals at site $x$ coalesce at a time. 

The moment duality (\ref{eq:moment_duality_coal}) can be used to quickly settle the question whether typically the underlying two types of individuals can coexist when the migration matrix $m$ is such that starting from finitely many particles, the process $(\xi_t)$ will almost surely end up with a single particle due to coalescence, that is 
\[
 |\xi_t|\stackrel{\rm def}{=}\sum_{x\in \Z^d}\xi_t(x)\xrightarrow[t\to\infty]{\rm a.s.}1.
\]
In this case, if we consider the homogenous initial condition $p_0(x)=1/2$ for all $x\in\Z^d$, then the duality \eqref{eq:moment_duality_coal} gives
convergence of all mixed moments:
\[
\E\left[\prod_{x\in\Z^d}p_t(x)^{\xi_0(x)}\right]\xrightarrow[t\to\infty]{}\frac{1}{2},\quad\forall\; \xi_0\in \Z_+^{\mathbb Z^d} \mbox{ with }\sum_{x\in \Z^d}\xi_0(x)<\infty,
\]
which implies 
\[
p_t\xrightarrow[t\to\infty]{{\rm (d)}}\tfrac{1}{2}\delta_{\mathbf 1}+\tfrac{1}{2}\delta_{\mathbf 0}.
\]
The two types cannot coexist. 

For the case $s>0$ and $\mu= 2$, the key idea from \cite{BEM} is to consider the transformed process: 
\[
\sigma_t\stackrel{\rm def}{=} 1- 2 p_t\in [-1,1]. 
\]
It satisfies the following system of SDEs:
\begin{linenomath*}\begin{align}
\label{eq:transformed_process}
 d\sigma_t(x)=\sum_{y\in \Z^d}m_{xy}\big(\sigma_t(y)-\sigma_t(x)\big)dt + \frac{s}{2} \big(\sigma_t(x)^3-\sigma_t(x)\big)dt - \sqrt{1-\sigma_t(x)^2}\,dW_t(x),\quad x\in \Z^d.
 \end{align}\end{linenomath*}
It is stated
in \cite[Lemma~2.1]{BEM} that the moment duality in (\ref{eq:moment_duality_coal}) still applies for $\sigma_t$ in place of $p_t$, but instead of a coalescing random walk, the dual process is now a {\bf branching annihilating random walk}.

\begin{defi}\label{def:DBARW}
Fix $s\geq 0$. Define a Markov process $(\xi_t)$ taking values $\xi_t\in \Z_+^{\Z^d}$ starting from finitely many particles at time zero (that is, $\sum_{x\in\Z^d}\xi_0(x)<\infty$) and with the following transition rates: 
\begin{linenomath*}\begin{align}\label{def:dbarw}
\begin{split}
&\mbox{migration:}\left\{
\begin{array}{ll}
\xi(x)\to \xi(x)-1\\
\vspace{-.4cm}\\
\xi(y)\to \xi(y)+1
\end{array}
\right.\mbox{ with rate }\xi(x)m_{xy},\\
&\hspace{-.9cm}\mbox{double branching: } \xi(x)\to \xi(x)+2\quad \mbox{ with rate }s\xi(x),
\\
&\hspace{.01cm}\mbox{annihilation: } \xi(x)\to \xi(x)-2\quad \mbox{ with rate }{\frac{\xi(x)[\xi(x)-1]}{2}}.
\end{split}
\end{align}\end{linenomath*}
This process is called a {\bf double branching annihilating random walk} (DBARW) with branching rate $s$. 
\end{defi}

\begin{lem}[\cite{BEM}]\label{lem:dual}
Fix $\mu= 2$ and $s\geq 0$. 
Then we have the following moment duality between the transformed process $(\sigma_t)$ defined by (\ref{eq:transformed_process}) and the DBARW 
$(\xi_t)$ from Definition~\ref{def:DBARW} with branching rate $s/2$: For all $\sigma_0\in [-1,1]^{\mathbb Z^d}$ and $\xi_0\in \Z_+^{\Z^d}$ with $\sum_{x\in\Z^d}\xi_0(x)<\infty$, 
\begin{align}\label{eq:momdual}\E_{\sigma_0}\left[\prod_{x\in \Z^d}\sigma_t(x)^{\xi_0(x)}\right]=\E_{\xi_0}\left[\prod_{x\in \Z^d} \sigma_0(x)^{\xi_t(x)}\right].\end{align}
\end{lem}

We will give the proof of Lemma~\ref{lem:dual} below in Section~\ref{sec:dual}.
As an application of the moment duality, we now show the equivalence of coexistence for the system \eqref{p:SDE} (for $\mu=2$) and survival of DBARW. See also \cite[Lemma 1]{SS} for a related result in the context of spin systems.

\begin{prop}\label{prop:equivalence}
Fix $s>0$. Let $(p_t)$ denote the solution to the system of SDEs in \eqref{p:SDE} for $\mu=2$, and let $(\xi_t)$ denote the DBARW with branching rate $s/2$ defined above. Then the following four statements are equivalent:
\begin{itemize}
\item[\rm (a)] For all initial conditions $p_0$ such that $p_0(x)\in(\varepsilon, 1-\varepsilon)$ for all $x$ for some small $\varepsilon>0$, we have long-term coexistence of $(p_t)$ with positive probability in the sense that there exists $\kappa\in (0,1)$ such that 
\[\liminf_{t\to\infty} \P\left(\kappa<p_t(0)<1-\kappa\right)>0.\]
\item[\rm (b)] There exists some initial condition $p_0$ for which we have long-term coexistence of $(p_t)$ with positive probability. 
\item[\rm (c)] 
The DBARW started with exactly two particles at the origin at time zero survives for all time with positive probability. That is, for the initial condition $\xi_0=2\1_{\{0\}}$, we have
\[\P\left(\xi_t\neq \mathbf 0,\;\forall\;t\geq 0\right)>0.\] 
\item[\rm (d)] 
The DBARW started with any even number of particles at time zero survives for all time with positive probability. 
\end{itemize}
\end{prop}
\begin{proof}
The equivalence of (c) and (d) is clear from the double branching mechanism in (\ref{def:dbarw}), and it is plain that (a) implies (b). Below, we show that (b) implies (c) and (c) implies (a). 
We work with the transformed process $\sigma_t=1-2p_t$ taking values in $[-1,1]$, which satisfies \eqref{eq:transformed_process}. Note that coexistence for $(p_t)$ as in (a) and (b) is equivalent to the following condition on the process $(\sigma_t)$. 
For all initial conditions $\sigma_0$ such that $\sup_{x\in \Z^d}|\sigma_0(x)|<1-2\varepsilon$ for some small $\varepsilon>0$, there exists $\kappa>0$ such that 
\[\liminf_{t\to\infty}\P\left(|\sigma_t(0)|<1-2\kappa\right)>0.\]

The proof that (b) implies (c) is as follows. Suppose by way of contradiction that (c) is false. 
This means that if we start the DBARW with two particles at the origin at time zero, that is $\xi_0=2\1_{\{0\}}$, then the process will die out with probability one, that is 
there exists $t>0$ with $\xi_t=\mathbf 0$ almost surely.  
Consider 
any initial condition $p_0$ for which the process $(p_t)$ coexists. 
Then by the moment duality and dominated convergence, we have
\begin{linenomath*}\begin{align}\label{sigmat}
\E\left[\sigma_t(0)^2\right]=\E\left[\prod_{x\in\Z^d}\sigma_t(x)^{\xi_0(x)}\right]
=\E\left[\prod_{x\in \Z^d}\sigma_0(x)^{\xi_t(x)}\right]\xrightarrow[t\to\infty]{}1.
\end{align}\end{linenomath*}
On the other hand, by the assumed long-term coexistence in (b), we know that there is some $\kappa\in (0,1)$ such that 
\[\delta\stackrel{\rm def}{=}\liminf_{t\to\infty} \P\left(\kappa<p_t(0)<1-\kappa\right)=\liminf_{t\to\infty}\P\left(|\sigma_t(0)|<1-2\kappa\right)>0,
\] 
which implies
\begin{linenomath*}\begin{align*}
\E\left[\sigma_t(0)^2\right]=&\E\left[\sigma_t(0)^{2}\left(\1_{\{|\sigma_t(0)|<1-2\kappa)\}} + \1_{\{|\sigma_t(0)|\ge1-2\kappa)\}}\right)\right]\\
\le &(1-2\kappa)^2\,\P\left(|\sigma_t(0)|<1-2\kappa\right) + \P\left(|\sigma_t(0)|\ge1-2\kappa\right)
\end{align*}\end{linenomath*}
and thus
\[\liminf_{t\to\infty}\E\left[\sigma_t(0)^2\right]\le(1-2\kappa)^2\,\delta + (1-\delta)<1,\]
which is a contradiction to (\ref{sigmat}). We have proved that (b) implies (c). 

Next, we prove that (c) implies (a). Assume that (c) holds. Consider $\varepsilon>0$ and any `permissible' initial condition $\sigma_0$ in (a), that is $|\sigma_0(x)|\le1-2\varepsilon$ for all $x\in\Z^d$. Suppose by way of contradiction that for all $\kappa>0$ we have
\[\liminf_{t\to\infty}\P\left(|\sigma_t(0)|<1-2\kappa\right)=0.\]
Then for any $\kappa>0$ we get
\begin{linenomath*}\begin{align*}\E\left[\sigma_t(0)^2\right]&=\E\left[\sigma_t(0)^{2}\left(\1_{\{|\sigma_t(0)|<1-2\kappa)\}} + \1_{\{|\sigma_t(0)|\ge1-2\kappa)\}}\right)\right]\\
&\ge 0 + (1-2\kappa)^2\,\P\left(|\sigma_t(0)|\ge1-2\kappa\right)\\
&= (1-2\kappa)^2\,\big(1-\P\left(|\sigma_t(0)|<1-2\kappa\right)\big),
\end{align*}\end{linenomath*}
from which
we infer that
\[\limsup_{t\to\infty}\E\left[\sigma_t(0)^2\right]\ge (1-2\kappa)^2\,\left(1-\liminf_{t\to\infty}\P\left(|\sigma_t(0)|<1-2\kappa\right)\right)=(1-2\kappa)^2.\]
Since $\kappa>0$ was arbitrary, this implies
\[\limsup_{t\to\infty}\E\left[\sigma_{t}(0)^2\right]=1.\]
On the other hand, since the DBARW survives by (c), we have $\delta=\P\left(\xi_t\neq \mathbf 0,\;\forall\;t\geq 0\right)>0$. Then by the duality we have for all $t>0$ 
\begin{linenomath*}\begin{align*}
0\le &\E\left[\sigma_t(0)^2\right]=\E\left[\prod_{x\in \Z^d}\sigma_t(x)^{\xi_0(x)}\right]=\E\left[\prod_{x\in \Z^d}\sigma_0(x)^{\xi_t(x)}\right]\\
\le &\E\left[\prod_{x\in \Z^d}|\sigma_0(x)|^{\xi_t(x)}\right]\\
=&\E\left[\prod_{x\in \Z^d}|\sigma_0(x)|^{\xi_t(x)}\left(\1_{\{(\xi_t)\text{ survives}\}}+\1_{\{(\xi_t)\text{ dies out}\}}\right)\right]\\
\le&(1-2\varepsilon)\,\delta+ 1-\delta<1,
\end{align*}\end{linenomath*}
again giving a contradiction. We have proved that (c) implies (a). The proof is complete.
\end{proof}

By Proposition~\ref{prop:equivalence}, coexistence for just one initial condition $p_0$ implies coexistence for \emph{all} initial conditions of the form considered in (a).
Hence, to determine coexistence, it suffices to work with nice initial conditions such as  $p_0(x)=1/2$ for all $x\in\Z^d$, which we considered when discussing the case $s = 0$ above. Recall that using a comparison with oriented percolation, \cite{BEM} proves coexistence for sufficiently large values of $s$ and so, in view of Proposition~\ref{prop:equivalence}, they obtain as a corollary the survival of DBARW if the branching rate is large enough.

There is another range of parameters for which solutions to (\ref{p:SDE}) admit nice duals and so the question whether coexistence occurs can be settled  again under appropriate conditions. 

\begin{lem}\label{lem:BCRW} 
For $\mu\in[-1,0]$ and $s\le0$, the solution $(p_t)$ to \eqref{p:SDE} is dual to a {\bf branching coalescing random walk} (BCRW) $(\xi_t)$ taking values in $\Z_+^{\Z^d}$ with transition rates given by (\ref{def:crw}), and in addition, the following two: 
\begin{linenomath*}\begin{align*}
&\xi(x)\to \xi(x)+1\quad\mbox{ with rate $(-s)(\mu+1)\xi(x)$,} \\
&\xi(x)\to \xi(x)+2\quad \mbox{ with rate $(-s)(-\mu)\xi(x)$}.
\end{align*}\end{linenomath*}
\end{lem}

Clearly, BCRW survives for all time with probability one whenever the initial condition is not trivial. 
Note that Lemma~\ref{lem:BCRW} contains the coalescing random walk dual for the stepping stone model (that is the model with $s = 0$) as a special case. We will also sketch the proof of Lemma~\ref{lem:BCRW} in Section~\ref{sec:dual}.

\begin{prop}\label{prop:extinction}
Let $(p_t)$ denote the solution to (\ref{p:SDE}) with parameters $s<0$ and $\mu\in[-1,0]$, and let $(\xi_t)$ denote the BCRW with branching rates defined in Lemma~\ref{lem:BCRW}. 
Assume that for each (non-trivial) initial condition $\xi_0$ with $0<\sum_{x\in \Z^d}\xi_0(x)<\infty$, we have
\begin{align}\label{BCRW-drift-to-infty}
|\xi_t| = \sum_{x\in \Z^d}\xi_t(x)\xrightarrow[t\to\infty]{\rm a.s.}+\infty.
\end{align}
Then for each initial condition $p_0$ such that $p_0(x)<1-\varepsilon$ for all $x\in\Z^d$ and some $\varepsilon>0$, it holds that  
\[
p_t\xrightarrow[t\to\infty]{\rm a.s.}\mathbf 0.
\]
\end{prop}
\begin{proof} 
By Lemma~\ref{lem:BCRW}, we have 
\begin{linenomath*}\begin{align*}
\E\left[\prod_{x\in \Z^d}p_t(x)^{\xi_0(x)}\right]&=\E\left[\prod_{x\in \Z^d}p_0(x)^{\xi_t(x)}\right]
\le\E\left[(1-\varepsilon)^{|\xi_t|}\right]\xrightarrow[t\to\infty]{}0
\end{align*}\end{linenomath*}
by dominated convergence.
This shows that all mixed moments of $(p_t)$ converge to zero, which is enough for our assertion. 
\end{proof}

Of course, this result is not surprising given that under the choice of the parameters $\mu$ and $s$ in Proposition \ref{prop:extinction}, type $0$ has a selective disadvantage. 
We note that condition \eqref{BCRW-drift-to-infty} is typically satisfied, see \cite[Thm. 4.2]{SU}. 
Since $(1-p_t)$ satisfies \eqref{p:SDE} with the same migration matrix $m$ and $s$ resp. $\mu$ replaced by $(-s)(1-\mu)$ resp. $\frac{\mu}{\mu-1}$, 
we may conclude that $p_t\xrightarrow[t\to\infty]{\rm a.s.}\mathbf 1$ for the parameter regime $s>0$ and $\mu\in[0,\frac{1}{2}]$.

\subsection{Proof of the moment duality}\label{sec:dual}
In this subsection, we give the proofs of the two duality results in Lemmas \ref{lem:dual} and \ref{lem:BCRW}.\\

\begin{proof}[Proof of Lemma \ref{lem:dual}]
Let $\mathsf L^\xi$ denote the generator of the DBARW $(\xi_t)$ from Definition \ref{def:DBARW} and $\mathsf L^\sigma$ denote the generator of $(\sigma_t)$ from \eqref{eq:transformed_process}. The bivariate duality function in use is now given by
\[
H(\sigma,\xi)=\sigma^\xi\stackrel{\rm def}{=}\prod_{x\in \Z^d}\sigma(x)^{\xi(x)}. 
\]
Then  (\ref{eq:transformed_process}) implies that 
\begin{linenomath*}\begin{align}
\mathsf L^\sigma H(\cdot,\xi)(\sigma)&=\sum_{x\in \Z^d}\left(\sum_{y\in \Z^d}m_{xy}\big(\sigma(y)-\sigma(x)\big) \right)\frac{\partial}{\partial \sigma(x)}\sigma^\xi\notag\\
&\quad+ \frac{s}{2} \sum_{x\in \Z^d}\left(\sigma(x)^3-\sigma(x)\right)\frac{\partial}{\partial \sigma(x)}\sigma^\xi\notag\\
&\quad + \frac{1}{2}\sum_{x\in \Z^d} \left(1-\sigma(x)^2\right)\frac{\partial^2}{\partial \sigma(x)^2}\sigma^\xi
\notag\\
\begin{split}
&=\sum_{x\in \Z^d:\xi(x)\geq 1}\left(\sum_{y\in \Z^d}m_{xy}\big(\sigma(y)-\sigma(x)\big) \right) \xi(x)\sigma(x)^{\xi(x)-1}\prod_{y\neq x}\sigma(y)^{\xi(y)} \notag\\
&\quad+ \frac{s}{2} \sum_{x\in \Z^d:\xi(x)\geq 1}\left(\sigma(x)^3-\sigma(x)\right) \xi(x)\sigma(x)^{\xi(x)-1}\prod_{y\neq x}\sigma(y)^{\xi(y)}\notag\\
&\quad+\frac{1}{2}\sum_{x\in \Z^d:\xi(x)\geq 2} \left(1-\sigma(x)^2\right)\xi(x)(\xi(x)-1) \sigma(x)^{\xi(x)-2}\prod_{y\neq x}\sigma(y)^{\xi(y)}
\end{split}
\notag\\
\begin{split}\label{Lsigma}
&=\sum_{x\in \Z^d:\xi(x)\geq 1} \xi(x)\left(\sum_{y\in \Z^d}m_{xy}\big(\sigma(y)\sigma(x)^{-1}-1\big) \right)\sigma^\xi \\
&\quad+ \frac{s}{2} \sum_{x\in \Z^d:\xi(x)\geq 1}\xi(x)\left(\sigma(x)^{2}-1\right) \sigma^\xi\\
&\quad+\frac{1}{2}\sum_{x\in \Z^d:\xi(x)\geq 2}\xi(x)(\xi(x)-1)\left( \sigma(x)^{-2} - 1\right) \sigma^\xi.
\end{split}
\end{align}\end{linenomath*}
On the other hand, the definition of $(\xi_t)$ shows that 
\begin{linenomath*}\begin{align*}
\mathsf L^\xi H(\sigma,\cdot)(\xi)&= \sum_{x\in \Z^d}\sum_{y\in \Z^d}\xi(x)m_{xy}\big(H(\sigma,\xi-\delta_x+\delta_y)-H(\sigma,\xi)\big)\\
&\quad+\frac{s}{2}\sum_{x\in \Z^d}\xi(x) \big(H(\sigma,\xi+2\delta_x)-H(\sigma,\xi)\big)
\\
&\quad+\frac{1}{2}\sum_{x\in \Z^d} \xi(x)\big(\xi(x)-1\big)\big(H(\sigma,\xi-2\delta_x)-H(\sigma,\xi)\big)\\
&= \sum_{x\in \Z^d:\xi(x)\geq 1}\sum_{y\in \Z^d}\xi(x)m_{xy}\big(\sigma(y)\sigma(x)^{-1}-1\big)\sigma^\xi\\
&\quad+\frac{s}{2}\sum_{x\in \Z^d:\xi(x)\geq 1}\xi(x) \big(\sigma(x)^2-1\big)\sigma^\xi\\
&\quad+\frac{1}{2}\sum_{x\in \Z^d:\xi(x)\geq 2} \xi(x)\big(\xi(x)-1\big)\big(\sigma(x)^{-2}-1\big) \sigma^\xi,
\end{align*}\end{linenomath*}
which is equal to the right-hand side of (\ref{Lsigma}). 
Hence, by the arguments in \cite[pages~188--193]{EK:MP} (see also \cite[Prop. 1.2]{JK14}), the required moment duality (\ref{eq:momdual}) holds. 
\end{proof}

\begin{proof}[Proof of Lemma \ref{lem:BCRW}]
In this proof, $\mathsf L^\xi$ denotes the generator of the BCRW $(\xi_t)$ with transition rates as defined in the statement of Lemma~\ref{lem:BCRW} and $\mathsf L^p$ denotes the generator of $(p_t)$ from \eqref{p:SDE}. 
The duality function is now changed to 
\[
H(p,\xi)=p^\xi\stackrel{\rm def}{=}\prod_{x\in \Z^d}p(x)^{\xi(x)}.
\] 
 Then by analogous calculations as in the proof of Lemma \ref{lem:dual}, we see that

\newpage
 \begin{align*}
 \mathsf L^p H(\cdot,\xi)(p) &= \sum_{x\in \Z^d:\xi(x)\geq 1}\sum_{y\in \Z^d}\xi(x)m_{xy}\big(p(y)p(x)^{-1}-1\big)p^\xi\\
 &\quad+s\sum_{x\in \Z^d:\xi(x)\ge1}\xi(x) \big(1-(\mu+1)p(x)+\mu p(x)^2\big)p^\xi \\
 &\quad+\frac{1}{2}\sum_{x\in \Z^d:\xi(x)\geq 2} \xi(x)\big(\xi(x)-1\big)\big(p(x)^{-1}-1\big) p^\xi\\
 &= \sum_{x\in \Z^d:\xi(x)\geq 1}\sum_{y\in \Z^d}\xi(x)m_{xy}\big(p(y)p(x)^{-1}-1\big)p^\xi\\
 &\quad+(-s)(-\mu)\sum_{x\in \Z^d}\xi(x) \big(p(x)^2-1\big)p^\xi+(-s)(\mu+1)\sum_{x\in \Z^d}\xi(x) \big(p(x)-1)\big)p^\xi\\
 &\quad+\frac{1}{2}\sum_{x\in \Z^d:\xi(x)\geq 2} \xi(x)\big(\xi(x)-1\big)\big(p(x)^{-1}-1\big) p^\xi\\
 &=\mathsf L^\xi H(p,\cdot)(\xi).
 \end{align*}
\end{proof}

\end{document}